\documentclass{amsart}
\usepackage{amscd}
\usepackage{amssymb}
\usepackage{cite}
\usepackage{amsmath}
\usepackage{enumerate}

\usepackage[usenames, dvipsnames]{color}
\usepackage{color}
\usepackage{hyperref}					
\hypersetup{colorlinks,
	linkcolor=blue,%
	citecolor=blue}

  \usepackage{mathrsfs}

\newtheorem{theorem}{Theorem}[section]
\newtheorem{thm}[theorem]{Theorem}
\newtheorem{lemma}[theorem]{Lemma}
\newtheorem{lem}[theorem]{Lemma}

\newtheorem{cor}[theorem]{Corollary}
\newtheorem{quest}[theorem]{Question}

\newtheorem{proposition}[theorem]{Proposition}

\newtheorem{rem}[theorem]{Remark}

\newtheorem{defn}[theorem]{Definition}
\newcommand{\cL}{\mathcal{L}}
\newcommand{\cS}{\mathcal{S}}
\newcommand{\cH}{\mathcal{H}}
\newcommand{\cF}{\mathcal{F}}

\newcommand{\norm}[1]{\left\lVert#1\right\rVert}

\newcommand{\tr}{{\rm Tr}}
\numberwithin{equation}{section}

\begin{document}
\title[Lack of isomorphic embeddings]{Lack of isomorphic embeddings of symmetric function spaces into operator ideals}


\author[S. Astashkin]{S. Astashkin}
\address[Sergei Astashkin]{Department of Mathematics, Samara National Research University, Moskovskoye
shosse 34, 443086, Samara, Russia  \emph{E-mail~:} {\tt astash56@mail.ru
 }}

\author[J. Huang]{J. Huang}
\address[Jinghao Huang]{School of Mathematics and Statistics, University of New South Wales, Kensington, 2052, NSW, Australia  \emph{E-mail~:} {\tt
jinghao.huang@unsw.edu.au}}

\author[F. Sukochev]{F. Sukochev}
\address[Fedor Sukochev]{School of Mathematics and Statistics, University of NSW, Sydney,  2052, Australia \emph{E-mail~:} {\tt f.sukochev@unsw.edu.au}}

\thanks{The work of the first author was completed as a part of the implementation of the development program of the Scientific and Educational Mathematical Center Volga Federal District, agreement no. 075-02-2020-1488/1.
The second  author was partially supported by  an Australian Mathematical Society Lift-Off Fellowship.
The   third   author was supported by the Australian Research Council  (FL170100052)}


\subjclass[2010]{46A50, 46E30, 47B10. \hfill Version~: \today.}


\begin{abstract}
Let $E(0,1)$ be a symmetric space on $(0,1)$ and $C_F$ be a symmetric ideal of compact operators on the Hilbert space $\ell_2$ associated with a symmetric sequence space $F$.
We give several criteria for $E(0,1)$ and $ F$ so that
 $E(0,1)$ does not embed into the ideal $C_F$,    extending the result
  for  the case when $E(0,1)=L_p(0,1)$ and $F=\ell_p $, $1\le p<\infty$,
    due to Arazy and Lindenstrauss~\cite{AL}.
\end{abstract}
\maketitle

\section{Introduction}

This paper has been motivated by a  beautiful  result due to Arazy and Lindenstrauss \cite[Theorem 6]{AL} (see also its antecedent \cite[Theorem 6.1]{LP}) that $$L_p(0,1)\not\hookrightarrow C_p,~2<p<\infty,$$
where $C_p$ is the Schatten $p$-class of compact operators on a separable  Hilbert space~$\cH$,
and the notation $A \hookrightarrow B$ (resp. $A\not\hookrightarrow B$)
 stands for indication that
  a Banach space $A$ is (resp. not) isomorphic to a subspace of a Banach space $B$.

The study of isomorphic classification of classical  Banach spaces  has a long history and it  is one of the most essential  topics in the theory of Banach spaces.
It is well-known that  $\ell_p \hookrightarrow L_p(0,1)$ if $1\le p <\infty$ \cite[Lemma~5.1.1 and Proposition~6.4.1]{AK} (see also  \cite{KP} and \cite{Odell}).
On the other hand,   $L_p(0,1)\hookrightarrow \ell_q$, $p\in [1,\infty)$,   $q\in [1,\infty)$, if and only if $p=q=2$ \cite[Ch.~XII, Theorem~9]{Banach} (see also \cite{JO,S01}). The above-mentioned result by
Arazy and Lindenstrauss
\cite{AL} can be viewed as a   noncommutative counterpart  of the latter fact.

For the deep theory concerning symmetric structure
of general symmetric function spaces $E(0,1)$/$E(0,\infty)$ and  symmetric sequence spaces $F$, we refer to outstanding monographs \cite{LT1, LT2,JMST}.
Let $E(0,1)$ be a symmetric space on $(0,1)$ and
 $F$ be a separable  symmetric sequence space, and let
 $C_F$ be a symmetric ideal of compact operators on the Hilbert space $\ell_2$,  generated by $F$. We are interested in the question:
 \begin{center}
   \emph{does  $E(0,1)$ isomorphically embed into $ C_F$?}
 \end{center}
In this general setting, the situation becomes dramatically different.
Consider, for instance, the sequence spaces $\ell_{p,q}$ (resp. function spaces $L_{p,q}(0,1)$), $1<p<\infty$, $1\le q<\infty$, which are  the most natural generalizations of the $\ell_p$-spaces (resp. $L_p(0,1)$-spaces). It was shown recently in \cite{KS} and  \cite{SS2018} that
$$\ell_{p,q}\not\hookrightarrow L_{p,q}(0,1), ~ p\in (1, \infty), ~q\in [1,\infty ), ~p\ne q, $$
which is in strong contrast with the fact  $\ell_p\hookrightarrow L_p(0,1)$   mentioned above.
This remark shows that
the techniques used by Arazy and Lindentrauss may not be sufficient to treat the general case.

Below, we briefly introduce the structure of the present paper.

In Section  \ref{pr}, we  provide all  necessary preliminaries and technical results. Some of them, known for Schatten $p$-classes, we establish for general  symmetric ideals.

It is well known that there are many fundamental differences in properties of the $L_p$-spaces in the cases when $1\le p<2$ and $2<p<\infty$.  The same observation holds also for symmetric spaces, which are located between the spaces $L_1(0,1)$ and $L_2(0,1)$,
on the one hand,
and between $L_2(0,1)$ and $L_\infty(0,1)$,
on the
other hand. In particular, the subspace structure of symmetric spaces located between $L_1(0,1)$ and $L_2(0,1)$ is much richer.
This fact stipulates the different approaches to  these two cases.

Recall that $L_p(0,1)$ has a subspace isomorphic to $\ell_r$ for any  $p<r\le  2$ \cite{JMST,LT2,AK,Dilworth}. Also, every subspace of the Schatten class $C_p$ has a subspace  isomorphic to $\ell_2$ or to $\ell_p$ \cite[Proposition  4]{AL}.
As a result, one can easily deduce the lack  of isomorphic embeddings of $L_p(0,1)$ into $C_p$ when $1\le p<2$ \cite[p. 197]{AL}. Arazy \cite[Corollary 3.2]{Arazy81} established the following deep result, which allows to use a similar reasoning for general ideals:
 \emph{
 for any $p\in [1,2)\cup (2,\infty )$ and  a separable symmetric sequence space $F$, $\ell_p$ is isomorphically embedded into the ideal $C_F$ generated by $F$ if and only if $\ell_p \hookrightarrow F$. } In Section \ref{<2}, we consider the case of symmetric spaces located between $L_1(0,1)$ and $L_2(0,1)$ and, by  making  use of this Arazy's result, show that
\emph{for any symmetric  space $E(0,1)$   such that $E(0,1)\supset L_p(0,1)$, with some $p<2$, we have
 $$
 E(0,1)\not\hookrightarrow C_F$$
 whenever a separable symmetric sequence space $F$ satisfies the condition: for every $\epsilon>0$ there exists $r\in (2-\epsilon,2)$ with
 $\ell_r\not\hookrightarrow F$}.
  In particular, we show that for any $1<  p<2$ and $1\le q<\infty$ we have $L_{p,q}(0,1)\not\hookrightarrow C_{p,q}:=C_{\ell_{p,q}}$.
Similarly,  $\Lambda _\psi^ q(0,1)\not\hookrightarrow C_{\lambda _w^{q'}}$, where $\Lambda _\psi^q(0,1)$ is an arbitrary Lorentz function space such that $\int_0^1 t^{-q/p}d\psi(t) <\infty$ for some $1<p<2$ and $\lambda _w^{q'}$ is any Lorentz sequence space.

Section \ref{first} contains the principal results of the paper (see Propositions \ref{lem:4} and \ref{prp main}). Here,
we consider symmetric spaces $E(0,1)$ located between the spaces $L_2(0,1)$ and $L_\infty(0,1)$
and operator ideals $C_F$ generated by $p$-convex and $q$-concave symmetric sequence spaces $F$, with some $2<p\le q<\infty$. In view of the classical Kadec--Pe{\l}czy\'{n}ski alternative for $L_p$, $p>2$~\cite{KP}, in this case we cannot hope on the existence of symmetric sequence spaces $G$ such that  $G\hookrightarrow E(0,1)$ and $G\not\hookrightarrow C_F$. In particular,
a recent deep result in \cite{HOS} shows that
a subspace of $L_p(0,1)$, $p>2$, either isomorphically embeds  into $\ell_p\oplus \ell_2$ or  contains $(\ell_2\oplus \cdots \oplus \ell_2)_{p}$.
However,  both $\ell_p\oplus \ell_2$ and $(\ell_2\oplus \cdots \oplus \ell_2)_{p}$ are isomorphic to some subspaces of $C_p$. This demonstrates why the case $p>2$ is much harder than the case $1\le p<2$.
Indeed, as one can see from  \cite[Theorem 6]{AL}, the proof of the lack of isomorphic embeddings of $L_p(0,1)$ into $C_p$, $p>2$, based on using the classical Haar basis, is rather complicated.
Observe that the idea of argument in \cite{AL} can be  traced back to  the proof of Theorem~6.1 by Lindenstrauss and Pe{\l}czy\'{n}ski in \cite{LP},
in which the authors stated that there are no  isomorphisms  from $L_p(0,1)$ into the space $(\ell_2\oplus\ell_2\oplus\cdots)_{p}$ but the proof there was  oversimplified and incomplete (see a related comment in \cite{AL}).
In Section \ref{first}, we succeed in extending  of  \cite[Theorem 6]{AL}
to
 some classes of operator ideals generated by $p$-convex and $q$-concave symmetric sequence spaces, $2<p\le q<\infty$, in particular, to the class of distributionally concave spaces.

In Section~\ref{app},  we collect applications of the results obtained in the previous section.
In particular,
in Corollary \ref{cor: Orl},
we prove that $L_{M}[0,1]\not\hookrightarrow  C_{\ell_M}$ for every submultiplicative
Orlicz function $M$, which is equivalent to a  $p$-convex Orlicz function for some $p>2$. 
This result can be treated as  a  partial noncommutative extension   of a well-known theorem by Lindenstrauss and Tzafriri that an Orlicz function space $L_M(0,1)$, which is not isomorphic to a Hilbert space, is not isomorphically embedded into any separable sequence Orlicz space $\ell_N$ \cite[Theorem~3]{LT3}. Another application of the results obtained in
the previous section relates to the Lorentz spaces $L_{p,q}$: we show that $L_{p,q}(0,1)\not\hookrightarrow C_{p,q}$ if $2<q\le p <\infty$ (see Theorem \ref{Lor}).


In the final section of the paper, we focus on considering the
 spaces $L_{2,q}$, $1\le q<\infty$, which do not satisfy the assumptions on symmetric function spaces in the preceding sections.
The space $L_{2,q}$, $1\le q<\infty$ is a  typical example  of a symmetric function space which is  ``very close''\:to the space $L_2$. We show that
$$L_{2,q}(0,1)\not\hookrightarrow C_{2,q}, ~q\in [1, \infty ), ~q\ne 2.$$
The main tools here are known properties of sequences of independent functions in $L_{2,q}$-spaces   \cite{AS08,CD89}, combined with
recent results on the lack of isomorphic embeddings from $\ell_{p,q}$  into $L_{p,q}(0,1)$   co-authored  by the third named  author \cite{KS,SS2018} and with  a result due to  Arazy \cite[Theorem 2.4]{Arazy81} describing shell-block basic  sequences in $C_F$.

\section{Preliminaries and auxiliary results}\label{pr}

 We use \cite{AK, LT1, LT2, JMST} as  main references to Banach space theory.
General facts concerning operator ideals may
be found in \cite{DP2,LSZ,Kalton_S, DPS2016,KPS} and references therein.
For convenience of the reader, some of the basic
definitions are recalled.

\subsection{Symmetric function and sequence spaces}
\label{spaces}

Let $L_0:=L_0(I)$ be the space of finite almost everywhere Lebesgue measurable functions either on $I=[0,1]$ or $I=[0,\infty)$
(with identification $m$-a.e.) equipped with  Lebesgue measure $m$ or the set $I=\mathbb{N}$ of all positive integers equipped with the counting measure (in the latter case, the space $L_0$ coincides with the space $\ell_\infty (\mathbb{N})$ of all bounded real-valued sequences).
Denote by $\cS : =S (I)$ the subset of $L_0$ which
consists of all functions (or sequences) $f$ such that {\it the distribution function}
$$
d_{f}(s):=m\left(\{t\in I:\,|f(t)|>s\}\right )$$ is finite for some $s>0$.
Any two functions $f$ and $g$ from $\cS$ are said to be {\it equimeasurable}
if $d_f(s)=d_g(s)$ for every $s>0.$
We denote by $f^*$ the non-increasing right-continuous rearrangement of $|f|$  given by
$$f^*(t):=\inf\left\{s\geq0:\ d_{f}(s)\le t\right \}, \ \ t\in I.$$

Let $E=E(I)$  be a Banach space of real-valued Lebesgue measurable functions if $I=[0,1]$ or $[0,\infty)$ (resp. of real-valued sequences if $I=\mathbb{N}$). To specify the notation we shall also use $E(0,1)$ or $E(0,\infty)$ instead of $E$. The space $E$ is
said to be an {\it ideal lattice} if the conditions $f\in E$ and $|g|\leq |f|,$ $g\in \cS$ imply that $g\in E$ and $\left\| g \right \|_E\leq \left\| f \right\|_E.$
The ideal lattice $E\subseteq \cS $ (respectively, $E\subset \ell_\infty$) is said to be a {\it symmetric function space} (respectively, \emph{symmetric sequence space}) if the norm $\left\|\cdot \right\|_E$ is {\it symmetric} (or {\it rearrangement invariant}),
that is, for every $f\in E$ and each function $g\in \cS$ (respectively, each sequence $g\in \ell_\infty$) with $g^*=f^*$, we have $g\in E$ and $\left\|g\right\|_E=\left\|f\right\|_E$ (see \cite{KPS, LT2}).

The function $\phi_{E}(t):=\left\|\chi_{[0,t]}\right\|_E$, $t\in I$ (resp. $\phi_{E}(n):= \left\| \sum_{k=0}^{n-1}  e_k \right\|_E$, $n\in\mathbb{N}$) is called the {\it fundamental function} of a  symmetric function (resp.  sequence) space $E$. In what follows, $\chi_{A}$ denotes the characteristic function of a set $A$ and $e_k$, $k=0,1,2\dots$, stand for the canonical unit vectors in a sequence space.

 For every symmetric function space $E$ its fundamental function $\phi_{E}$ is {\it quasi-concave}, that is, it is nonnegative, increases, $\phi_{E}(0)=0$, and the function $\phi_{E}(t)/t$ decreases on $I$. The fundamental function of a symmetric sequence space has analogous properties.

 Without loss of generality, for any symmetric function (resp. sequence) space $E$ we always assume that $\left\| \chi_{[0,1]} \right\|_E=1$ (resp. $\left\| e_0\right\|_E =1$).

If $\tau>0,$ the dilation operator $\sigma_{\tau}$ is defined by setting $\sigma_{\tau}f(s)=f(s/{\tau}),$ $s>0,$ in the case of the semi-axis. For a function on  the interval $(0,1),$ the operator $\sigma_{\tau}$ is defined by
$$
\sigma_{\tau}f(s)=
\begin{cases}
f(s/\tau),& s\leq\min\{1,\tau\}, \\
0,& otherwise.
\end{cases}
$$
The operator $\sigma_{\tau}$ is bounded in every function symmetric space $E(I)$ and $\left\|\sigma_\tau\right\|_{E(I)\to E(I)}\le\max(1,\tau)$ \cite[Theorem~II.4.5]{KPS}.
In particular, $\left\|\sigma_\tau\right\|_{L_p\to L_p}=\tau^{1/p}$, $1\le p\le\infty$.

Similarly, in the case of sequence spaces, for each $m \in \mathbb N$ by $ {\sigma}_m$ and $ {\sigma}_{1/m}$ we define the {\it dilation operators} as follows: if $a = (a_n)_{n=0}^\infty$, then
$$
{\sigma}_m a = \left( ({\sigma}_m a)_n \right)_{n=0}^{\infty}
= \big ( \overbrace {a_0, a_0, \ldots, a_0}^{m}, \overbrace {a_1, a_1, \ldots, a_1}^{m}, \ldots \big)
$$
and
$$
{\sigma}_{1/m} a =
\left( ( {\sigma}_{1/m} a)_n \right)_{n=0}^{\infty} = \Big (\frac{1}{m} \sum_{k=nm  }^{(n+1)m-1 } a_k \Big)_{n=0}^{\infty}
$$
(see, for example, \cite[p.~223]{KPS}). As in the case of function spaces, these operators are bounded in every symmetric sequence space $F$ with the same estimates for their norms.
Also, $\left\|\sigma_{1/m} \right\|_{\ell_p\to \ell_p}=m^{-1/p}$ and $\left\|\sigma_{m} \right \|_{\ell_p\to \ell_p}=m^{1/p}$, $m\in\mathbb{N}$, $1\le p\le\infty$.

The {\it dilation function} ${\mathcal M}_f$ of a nonnegative function $f$ on the interval $(0,1)$ is defined by
$$
{\mathcal M}_f(t):=\sup_{0<s\le\min(1,1/t)}\frac{f(st)}{f(s)},\;\;t>0.$$
 Since the function ${\mathcal M}_f$ is submultiplicative, there are the following {\it dilation exponents}:
\[
 \gamma_\varphi:= \lim\limits_{t\to+0}
 \frac{\ln M_\varphi(t)}{\ln t}\;\;\mbox{and}\;\;
 \delta_\varphi:= \lim\limits_{t\to\infty}
 \frac{\ln M_\varphi(t)}{\ln t}
\]
For each quasi-concave function $\varphi$ we have $0\le \gamma_\varphi\le \delta_\varphi\le 1$ \cite[\S\,II.1]{KPS}.

We say that $x\in S(I)$, where $I=[0,1]$ or $I=(0,\infty)$ (resp. $x=(x_k)_{k=0}^\infty \in \ell_\infty$), is submajorized by $y\in S(I)$ (resp.   by $y=(y_k)_{k=0}^\infty \in \ell_\infty$)  in the sense of {\it Hardy--Littlewood--P\'{o}lya} (briefly, $x\prec\prec y$) if
$$
\int_0^t x^*(s)\,ds\le \int_0^t y^*(s)\,ds,\;\;t\in I$$
(resp.
$$
\sum _{k=0}^n x_k^* \leq\sum_{k=0}^n y_k^*, n=0,1,2,\dots).$$
See   \cite[Definition~2.a.6 and Proposition~2.a.8]{LT2} or \cite[\S\,2.3]{Bennett_S} for the main properties of this pre-order.  Recall here only that the norm of every separable symmetric space $E$ is {\it monotone with respect to Hardy--Littlewood--P\'{o}lya submajorization}, i.e.,
if $x\in  \cS$ and $y\in E$ such that $x\prec \prec y  $, then $x\in E$ with $\norm{x}_E\le \norm{y}_E$ (see e.g.   \cite[Proposition~2.a.8]{LT2}).

Let $E=E(0,1)$ be a symmetric space, $x_{k,j}:=\phi_E(2^{-k})^{-1}\chi_{\Delta_k^j}$, where $\Delta_k^j:=[j2^{-k},(j+1)2^{-k})$, $k=1,2,\dots$, $j=0,1,\dots,2^k-1$. Then, for each $k=1,2,\dots$, $\{x_{k,j}\}_{j=0}^{2^k-1}$ is a normalized basis in the subspace $E_k:=[x_{k,j},j=0,\dots,2^k-1]$ of $E$.

\begin{lem}\label{finite dim}
Suppose that a symmetric space $E=E(0,1)$ satisfies the condition:
\begin{equation}
\label{lem dim}
\lim_{t\to\infty}\frac{{\mathcal M}_{\phi_E}(t)}{t^{1/2}}=0.
\end{equation}
Then, for arbitrary $\epsilon>0$ there exists a positive integer $k=k(\epsilon)$ such that for every linear operator $V:\,E_k\to \ell_2$ such that $\|V\|=1$ we have
$$
|\{j=0,1,\dots,2^{k}-1:\, \left\|V(x_{k,j})\right\|_{\ell_2}\ge\epsilon\}|\le \epsilon 2^{k}.$$
\end{lem}
\begin{proof}
Denote for each $k\in\mathbb{N}$
$$
A_k(\epsilon):=\{j=0,1,\dots,2^{k}-1:\,\left\|V(x_{k,j})\right\|_{\ell_2}\ge\epsilon\}.$$
Then, on the one hand, for any $\theta_j=\pm 1$, $j=0,1,\dots,2^{k}-1$,
\begin{eqnarray*}
\Big\|\sum_{j\in A_k(\epsilon)}\theta_jV(x_{k,j})\Big\|_{\ell_2}&=& \Big\|V\Big(\sum_{j\in A_k(\epsilon)}\theta_jx_{k,j}\Big)\Big\|_{\ell_2}\le \Big\|\sum_{j\in A_k(\epsilon)}\theta_jx_{k,j}\Big\|_{\ell_2}\\
&=& \frac{\phi_E(|A_k(\epsilon)|\cdot 2^{-k})}{\phi_E(2^{-k})}\le {\mathcal M}_{\phi_E}(|A_k(\epsilon)|).
\end{eqnarray*}
Hence, we have
$$
\Big({\rm Ave}_{\theta_j=\pm 1}\Big\|\sum_{j\in A_k(\epsilon)}\theta_jV(x_{k,j})\Big\|_{\ell_2}^2\Big)^{1/2}\le {\mathcal M}_{\phi_E}(|A_k(\epsilon)|).$$
On the other hand, according to the parallelogram identity,
$$
\Big({\rm Ave}_{\theta_j=\pm 1}\Big\|\sum_{j\in A_k(\epsilon)}\theta_jV(x_{k,j})\Big\|_{\ell_2}^2\Big)^{1/2}=\Big(\sum_{j\in A_k(\epsilon)}\left\|V(x_{k,j})\right\|_{\ell_2}^2\Big)^{1/2}\ge \epsilon |A_k(\epsilon)|^{1/2}.$$
Consequently,
$$
\epsilon |A_k(\epsilon)|^{1/2}\le {\mathcal M}_{\phi_E}(|A_k(\epsilon)|),\;\;k=1,2,\dots$$
Combining this inequality with the hypothesis of the lemma, we conclude that there is a constant $C=C(\epsilon)$ such that $|A_k(\epsilon)|\le C$ for all $k=1,2,\dots$. Choosing $k$ so that $\epsilon 2^{k}>C$, we get the desired result.
\end{proof}

\begin{rem}
\label{rem lem dim}
In particular, by the definition of the dilation exponents, one can readily see that condition \eqref{lem dim} follows from the inequality $\delta_{\phi_E}<1/2$.
\end{rem}


\subsection{Orlicz and Lorentz spaces}
\label{deflor}

The most known and important symmetric spaces are the $L_p$-spaces, $1\le p\le\infty$. Their natural generalization is the Orlicz spaces. Let $M$ be an Orlicz function, that is, an increasing convex function on $[0, \infty)$ such that $M(0) = 0$. Denote by $L_{M}:=L_M(I)$, where $I=(0,1)$ or $(0,\infty)$, the {\it Orlicz space}
on $I$ (see e.g. \cite{KR}) endowed with the Luxemburg--Nakano norm
$$
\left\| f \right\|_{L_{M}} = \inf \left\{v > 0 \colon \int_I M(|f(t)|/v) \, dt \leq 1 \right \}.
$$
In particular, if $M(u)=u^p$, $1\le p<\infty$, we obtain $L_p$. One can readily check that the fundamental function of $L_M$ is determined by the formula: $\phi_{L_M}(u)=1/M^{-1}(1/u)$, $0<u\le 1$, where $M^{-1}$ is the inverse function for $M$.


Similarly, we can define an {\it Orlicz sequence space}. Specifically, the space $\ell_{N}$, where $N$ is an Orlicz function, consists of all sequences $(a_{k})_{k=0}^{\infty}$ such that
$$
\left\| (a_{k})_{k=0}^{\infty} \right \|_{\ell_{N}} := \inf\left\{u>0: \sum_{k=0}^{\infty} N\Big( \frac{|a_{k}|}{u} \Big)\leq 1\right\}<\infty.
$$

An Orlicz function $H$ satisfies the $\Delta_{2}^{\infty}$-condition ($H\in \Delta_{2}^{\infty}$) (resp. the $\Delta_{2}^{0}$-condition ($H\in \Delta_{2}^{0}$)) if
$$
\limsup\limits_{t \to \infty}\frac{H(2t)}{H(t)} < \infty\;\;\mbox{(resp.}\;\;\limsup\limits_{t \to 0}\frac{H(2t)}{H(t)} < \infty).
$$
It is well known that an Orlicz function space $L_{M}$ on $[0,1]$ (resp. an Orlicz sequence space $\ell_{N}$) is separable if and only if  $M\in \Delta_{2}^{\infty}$ (resp. $N\in \Delta_{2}^{0}$).

Observe that the definition of an Orlicz sequence space $\ell_{N}$ depends (up to equivalence of norms) only on the behaviour of the function $N$ near zero. More precisely, in the separable case (i.e., when $N,N_1 \in \Delta_{2}^{0}$), the following conditions are equivalent: (1) $\ell_{N}=\ell_{N_1}$ (with equivalence of norms); 2) the unit vector bases of the spaces $\ell_{N}$ and $\ell_{N_1}$ are equivalent; 3) there are $C > 0$ and $t_{0} > 0$ such that for all $0 \leq t \leq t_{0}$ it holds
$$
C^{-1}N_1(t) \leq N(t) \leq CN_1(t)
$$
(cf. \cite[Proposition~4.a.5]{LT1}).
Quite similarly, the definition of an Orlicz function space $L_M$ on $[0,1]$ depends only on the behaviour of the function $M$ for large values of the argument.

Another natural generalization of the $L_p$-spaces is the class of Lorentz spaces.
Let $\psi$ be an increasing concave function on $I$ with
$\psi(0)= \psi(+0) = 0$, $\psi(\infty)=\infty$ and $1\le q<\infty.$  The Lorentz space $\Lambda^q_\psi:=\Lambda^q_\psi(I)$ consists of all measurable functions $f$ on $I$, for which
\[\norm{f}_{\Lambda^q_\psi }:= \Big(\int_If^*(t)^q\,d\psi(t)\Big)^{1/q} <\infty\]
(see \cite{KMP,KamMal,LT2,LT1}).
It is well-known that $\Lambda^q_\psi(I)$ is separable for all $\psi$ and $1\le q<\infty$~\cite{KamMal}.

 Recall also the definition of Lorentz spaces $L_{p,q}:=L_{p,q}(I)$ \cite{ON,Bennett_S,Creekmore,Dilworth}. If $1 <  p < \infty$ and $1 \leq q\le \infty,$ then
 $L_{p,q}$ is the space of all measurable functions $f$ on $I$ such that
$$\norm{f}_{p,q}:= \begin{cases}
 \left(\int_I f^*(t)^q d(t^{q/p})\right)^{1/q}, &q<\infty;  \\
 \sup _{t\in I }(t^{1/p}f^*(t)),   &q=\infty \\
\end{cases}$$
 is finite. In particular, $L_{p,\infty}$, $1<p<\infty$, are called often the {\it weak} $L_p$-spaces.
It is clear that if $1\le q\le p<\infty$ and
 $\psi(t):= t^{q/p}$, then  $L_{p,q}(I) = \Lambda^q_{\psi}(I)$.
In this case $\norm{\cdot}_{p,q}$ defines a norm under which $
L_{p,q}$ is a separable symmetric space; for $1<p<q\le \infty$,  $\norm{\cdot}_{p,q}$ is a quasi-norm which is known to be equivalent to a symmetric norm \cite[Theorem~4.4.6]{Bennett_S}.

Define also the {\it Lorentz sequence space} $\lambda_w^q$, where $1\le q<\infty$ and $w=(w_n)_{n=0}^\infty$ is a decreasing sequence of positive numbers such that $w_0=0$, $\lim_{n\to\infty}w_n=0$ and $\sum_{n=0}^\infty w_n=\infty$, as the set of all sequences $a=(a_n)_{n=0}^\infty$ such that
$$
\|a\|_{\lambda_w^q}:=\Big(\sum_{n=0}^\infty (a_n^*)^qw_n\Big)^{1/q}<\infty,$$ where $(a_n^*)_{n=0}^\infty$ is the decreasing permutation of the sequence $(|a_n|)_{n=0}^\infty$.

Similarly, if $1 <  p < \infty$ and $1 \leq q\le \infty,$ the space  $\ell_{p,q}$ consists of all sequences $a=(a_n)_{n=0}^\infty$, for which
$$
\norm{a}_{\ell_{p,q}}:= \begin{cases}
 \Big(\sum_{n=0}^\infty (a_n^*)^q (n^{q/p}-(n-1)^{q/p})\Big)^{1/q}, &q<\infty;  \\
 \sup _{n=0,1,2,\dots}(a_n^*n^{1/p}),   &q=\infty \\
\end{cases}$$
 is finite.

  We note that the norm of all Orlicz and Lorentz spaces defined above
is monotone with respect to Hardy--Littlewood--P\'{o}lya submajorization
\cite{KPS,Bennett_S,KamMal}.


\subsection{Operator ideals in $B(\cH)$}
 Define by $B(\cH)$  the $*$-algebra of all bounded linear operators acting on a separable infinite-dimensional  Hilbert space $\cH$.
For any $x\in B(\cH)$, we denote by
$\left\{ \mu(n;x) \right\}_{n=0}^\infty$
the sequence of  singular values of $x$, i.e., the eigenvalues of $(x^*x)^{1/2}$ arranged in a non-increasing ordering, counting multiplicity.

Let $F$ be a symmetric sequence space.
We work with the ideal $C_F$ in the algebra $B(\cH)$ defined as follows
$$C_F =\{a\in B(\cH):\ \{\mu(k;a)\}_{k=0}^\infty \in F\}.     $$
 This ideal becomes Banach when equipped with the norm
$\left\|a\right\|_{C_F}=\left\|\mu(a)\right \|_F,~ a\in C_F $\cite{LSZ,Kalton_S}.
When $F=\ell_p$, $p\ge 1$, we denote  by $C_p$ the corresponding operator ideal
$$ \left\{
a\in B(\cH):\ \mu(A)\in \ell_p\right\},\quad \left\|a \right\|_p=\left\| \mu(a ) \right\|_p,$$
which are the   best known examples of   Banach ideals in $B(\cH)$ (called Schatten-von Neumann $p$-class).
When  $F=\ell_{p,q}$, $1<p<\infty$, $1\le q\le\infty$, we denote the corresponding operator ideal  by $C_{p,q}$.

 Quite similarly, we can define also the ideals $C_{\ell_M}$, for each Orlicz function, and $C_{\lambda_{w}^q}$, where $w=(w_n)_{n=0}^\infty$ is a decreasing sequence of positive numbers such that $w_0=1$, $\lim_{n\to\infty}w_n=0$, $\sum_{n=0}^\infty w_n=\infty$ and $1\le q<\infty$.

We use the notion of the right support of the operator $a\in B(\cH)$ defined as follows
$$r(a)=\bigwedge \{ p \mbox{ is a projection in }B(\cH):  ap=a   \}.$$
Operators $a_k\in B(\cH),$ $k\geq0,$ are called disjointly supported from the right if $r(a_{k_1})r(a_{k_2})=0$ for $k_1\neq k_2.$ Equivalently, $|a_{k_1}|\cdot|a_{k_2}|=0$ for $k_1\neq k_2.$ Whenever $a_k,$ $k\geq0,$ are disjointly supported from the right and the subsets $\mathcal{A}_l\subset\mathbb{Z},$ $l\geq0,$ are disjoint (i.e. $\mathcal{A}_{\ell_1}\cap\mathcal{A}_{\ell_2}=\varnothing$ for $l_1\neq l_2$), the elements
$$b_l=\sum_{k\in\mathcal{A}_l}a_k$$
are also disjointly supported from the right.

Modifying  the argument in \cite[Proposition 2.3]{Suk.}, we observe that if
a basic  sequence  in $C_F$ consists of elements which are pairwise disjointly supported from the left and from  the right,
then
this sequence is (isometrically) equivalent  to  the corresponding basic sequence  of pairwise disjointly supported elements in the symmetric sequence space $F$.
However, in general,   this fact fails for  elements in $C_F$ which are   pairwise disjointly supported  only from the left (or, only from the right).

\subsection {Distributional concavity}
\label{Distributional concavity}

Fix a partition $\mathbb{N}=\bigcup_{k=1}^\infty U_k$, where $U_k$, $k=1,2,\dots,$ are infinite disjoint sets, and one-to-one mappings $\kappa_k:\,U_k\to \mathbb{N}$ , $k=1,2,\dots$. Define now as a {\it disjoint sum} of a set $a_k=(a_{k,i})_{i=1}^\infty$, $k=1,2,\dots,$ of sequences of real numbers the sequence
$$
\bigoplus_{k=1}^\infty a_k:=\sum_{k=1}^\infty \sum_{i\in U_k}a_{k,\kappa_k(i)}e_i,$$
where $e_i$ are standard unit vectors.
It is important to observe that the distribution function of a disjoint sum
$\bigoplus_{k=1}^\infty a_k$ does not depend on the particular choice of a partition $\mathbb{N}=\bigcup_{k=1}^\infty U_k$ and mappings $\kappa_k:\,U_k\to \mathbb{N}$ , $k=1,2,\dots$.

In particular, in the case when  $a_k=a$ if $k=1,\dots,n$, and $a_k=0$ if $k>n$, we will denote $
\bigoplus_{k=1}^\infty a_k$ by $a^{\oplus n}$.

Now we can adopt the well-known definition of distributionally concave symmetric function spaces on $[0,1]$ (see e.g. \cite[Definition 2.2]{ASW}) in the case of symmetric sequence spaces as follows.

\begin{defn}\label{definition-of-D-convexity}
A symmetric sequence space $F$ is called  {\it distributionally concave} if there is a constant $c_F>0$ such that for every finite collection $\{a_k\}_{k=1}^n\subset F$ we have
$$
\norm{\bigoplus_{k=1}^n a_k}_F \geq c_F \min_{1 \leq k \leq n} \norm{a_k^{\oplus n}}_F.
$$
\end{defn}

As in the case of function spaces (see e.g. \cite{MSS},  \cite[Proposition 2.5]{ASW}, \cite[Proposition 19]{LS} and \cite[Proposition 2.5]{S-RUC}), it can be easily checked that  all Orlicz sequence spaces  $\ell_M$ and Lorentz spaces $\lambda_w^q$ (in particular, $\ell_{p,q}$, $1\le q\le p<\infty$) are   distributionally concave.

\begin{rem}\label{ncdc}
For arbitrary Hilbert space $\cH$ and every positive integer $n$ we clearly have $\cH^{\oplus n} \simeq \cH$ a natural isomorphism. Then, if $a_k\in B(\cH)$, $k=1,2,\dots,n$, considering the image of the direct sum $\bigoplus_{k=1}^n  a_k$, under this isomorphism, as an element of $\cH$, by the definition of singular values of operators, we obtain
$$\mu\left(\bigoplus_{k=1}^n a_k \right)  =\left(\bigoplus_{k=1}^n \mu(a_k) \right)^*,$$
where $\bigoplus_{k=1}^n \mu(a_k)$ is the  disjoint sum of the sequences $\mu(a_k)$, $k=1,2,\dots,n$\footnote{This fact justifies that for the direct sum we use the same symbol as for a disjoint sum.} and $\left(\oplus_{k=1}^n \mu(a_k) \right)^*$ stands for the decreasing rearrangement of $\oplus_{k=1}^n \mu(a_k) $ (see Subsection~\ref{spaces}). Hence, for any distributionally concave symmetric sequence space $F$ we have
\begin{align}\label{ncdceq}
\norm{\bigoplus_{k=1}^n a_k}_{C_F} = \norm{\mu\left(\bigoplus_{k=1}^n  a_k \right)  }_F &\geq c_F \min_{1 \leq k \leq n} \norm{\mu(a_k)^{\oplus n}}_F \nonumber \\
&=c_F \min_{1 \leq k \leq n} \norm{a_k^{\oplus n}}_{C_F}.
\end{align}

\end{rem}


\subsection{The upper triangular part of $C_F$}

Recall first that a pair $(X_0,X_1)$ of Banach spaces is called a {\it
Banach couple} if $X_0$ and $X_1$ are both linearly and continuously
embedded in some Hausdorff linear topological vector space. In particular, every two symmetric (function or sequence) spaces $E_0$ and $E_1$ form a Banach couple.

A Banach space $X$ is called {\it interpolation} with respect to a Banach couple $(X_0,X_1)$ (in brief, $X\in{\rm Int}(X_0,X_1)$) whenever $X_0 \cap X_1 \subset X \subset X_0+X_1$ and each linear operator $T:\,X_0+X_1\to X_0+X_1$, which is bounded in $X_0$ and in $X_1$, is bounded in $X$. For a further information related to the theory of interpolation of operators we refer to the monographs \cite{Bennett_S,KPS,LT2}.

The following lemma establishes an isomorphic embedding from an operator ideal onto its upper triangular part, which extends \cite[Proposition 1]{AL}.
\begin{lem}\label{triangular lemma} For every separable symmetric sequence space $F\in{\rm Int}(\ell_p,\ell_q),$  $1<p,q<\infty,$ there exists an isomorphic embedding from $C_F$ onto its upper triangular part $U_{F}:=\{ x\in C_F : x_{ij}=0 , i>j \}$.
\end{lem}
\begin{proof} Let $T$ be the upper triangular truncation operator (see  e.g. \cite[(1.1)]{Ar78}).
Recall that $T$ is bounded on $C_F $ \cite[Corollary 4.12]{Ar78}.
Let $S$ be the transposition operator.
In fact, it is an isometry on $C_F$ because it preserves the singular value function.
Let $D$ be the diagonal cut.
Note that the assumption that $F\in {\rm Int}(\ell_p,\ell_q)$  implies that $F$ is monotone with respect to Hardy--Littlewood--P\'{o}lya  submajorization (see e.g. \cite{KPS} and \cite{Bennett_S}, see also  \cite[Theorem 3.1]{Cadilhac}).
Moreover, since  $Dx\prec\prec x$ for every $x$ (see e.g. \cite[Lemma 6.1]{DPS2016}), it follows that  $D$ is bounded on $C_F$.

Define a bounded mapping $A: C_F \to U_{F}\oplus U_{F}:=U_{F}^{\oplus 2}$ by the formula
\begin{align}\label{AXTX}
Ax=Tx\oplus TSx,\quad x\in C_F.
\end{align}
The boundedness of $A$ follows immediately from the boundedness of $T:C_F \to C_F.$ We now show that $A$ is an isomorphic embedding. Since $x=Tx+(STS)(x-Dx),$ it follows that
\begin{align}\label{xaxdx}
\left\|x\right\|_{C_F} \leq \left\|Tx \right\|_{C_F}+ \left\|(TS)(x-Dx) \right\|_{C_F}
&\leq \left\|Tx\right\|_{C_F} + \left\|TSx \right\|_{C_F}+ \left\|T\right\| \left\|Dx \right\|_{C_F} \nonumber\\
&\leq 2\left\|Ax \right\|_{U_{F}^{\oplus 2}}+\left\|T\right\| \left\|Dx \right\|_{C_F }.\end{align}
Since $Dx=DTx\prec\prec Tx$ and since the norm in $F$ is monotone with respect to the Hardy--Littlewood--P\'{o}lya submajorisation,  
it follows that
$$\left\|Dx\right\|_{C_F}\leq \left\|Tx\right\|_{C_F}\stackrel{\eqref{AXTX}}{\leq}\|Ax\|_{U_{F}^{\oplus 2}}.$$
Therefore, by \eqref{xaxdx},  we have
$$\left\|x\right\|_{C_F}\leq3 \left\|Ax\right \|_{U_{F}^{\oplus 2}},$$
which shows  that the bounded operator  $A:C_F \to U_{F}^{\oplus 2}$ is an isomorphic embedding.

We now claim that $U_{F}^{\oplus 2}$ admits an isomorphic embedding into $U_{F}.$ Let $p_1,p_2$ be projections in $B(\cH)$ such that
\begin{enumerate}
\item $p_1$ and $p_2$ are diagonal.
\item $p_1p_2=0.$
\item each $p_k$ has infinite rank, i.e., $\tr(p_k)=\infty$, where $\tr$ is the standard trace on $B(\cH)$.
\item $p_1+p_2={\bf 1} .$
\end{enumerate}
Since $p_k$ is diagonal, it follows that, for a given $l\geq0,$ either $p_ke_{ll}=e_{ll}$ or $p_ke_{ll}=0.$
Set $B_k=\{l\geq0: p_ke_{ll}=e_{ll}\}$, $k=1,2$.
Let $\theta_k:B_k\to\mathbb{Z}_+ \cup \{0\}$ be the monotone bijection, i.e.,
 $$ \theta_k(l)= \left| \{ j\in B_k:j  \le l\}\right|  -1 , ~l\in B_k  . $$
 Define the mapping $V_k:(p_kB(\cH)p_k,{\rm Tr})\to(B(\cH),{\rm Tr})$ by the setting
$$V_k:\sum_{l_1,l_2\in B_k}a_{l_1l_2}e_{l_1l_2}=\sum_{l_1,l_2\in B_k}a_{l_1l_2}e_{\theta_k(l_1)\theta_k(l_2)}.$$
This is a trace-preserving $*$-homomorphism.
Moreover,
$V_k$ preserves the singular value function. Since $\theta_k$ is monotone, it follows that $V_k$ maps upper triangular matrices to the upper triangular ones.
Thus, $V_k:p_kU_{F}p_k\to U_ F$ is an isometry.
 Hence,
$$U_{F}\approx p_1U_{F}p_1,\quad U_{F}\approx p_2U_{F}p_2,$$
$$U_{F}^{\oplus 2}\approx p_1U_{F}p_1\oplus p_2U_{F} p_2\approx p_1U_{F}p_1+p_2U_{F}p_2,$$
where the last equivalence can be seen as follows: for every $x_k\in p_kU_{F}p_k,$ we have $\mu(x_1+x_2)=\mu(x_1\oplus x_2)$ and, therefore, $\left\|x_1+x_2\right\|_{C_F}\approx \left\|x_1\right\|_{C_F}+\left\|x_2\right \|_{C_F}.$

Since we have  established that $C_F $ admits an isomorphic embedding into $U_{F}^{\oplus 2}$,
 it follows that $C_F$ admits an isomorphic embedding into $U_{F}.$  This completes the proof.
\end{proof}

\subsection{ Haar system and Rademacher functions}
\label{Haar, Rademacher}
Recall that the Haar system  \cite[Definition 1.a.4.]{LT1}  can be defined   
for $l=0, 1, \cdots, 2^k-1$ and $k=0,1,\cdots,$  by setting
$$
h_{2^k+l}(t)=
\begin{cases}
1,& l \cdot 2^{-k} \le t < (l+\frac 12)2^{-k}, \\
-1,& (l+\frac12)2^{-k}  \le t <  (l+1) 2^{-k},\\
0,& \mbox{otherwise.}
\end{cases}
$$

\begin{lem}\label{right disjoint} Let $F$ be a separable  symmetric  sequence space and  $F\in{\rm Int}(\ell_p,\ell_q)$  for some $1<p,q<\infty$.
If a separable space $E(0,1)$ isomorphically embeds into $U_{F},$ then there is another isomorphic embedding of $E(0,1)$ into $U_F$ such that images of the Haar basis are disjointly supported from the right.
\end{lem}
\begin{proof} Let $e_{ll},$ $l\geq0,$ be the $l$-th matrix unit on the diagonal. It is known  (see \cite[Lemma 4.5]{Ar78}) that $\{U_{F }e_{ll}\}_{l\geq0}$ is an unconditional finite dimensional decomposition (see e.g. \cite[Chapter 1.g]{LT1} for definition) in $U_{F }.$
By \cite[Theorem 5.2]{DFPS}, 
 there exists an increasing sequence $\{q_m\}_{m\geq0}\subset\mathbb{Z}_+$ and a basic sequence $z_m\in U_F (\sum_{l=q_m}^{q_{m+1}-1}e_{ll})$ which is equivalent to the Haar basis.
 It is immediate that the support $r(z_m)\leq \sum_{l=q_m}^{q_{m+1}-1}e_{ll},$ $m\geq0.$
 Therefore, $r(z_{m_1})r(z_{m_2})=0$ for $m_1\neq m_2.$
 Recall that the Haar system is a Schauder  basis of any separable symmetric function space on~$(0,1)$~\cite[Proposition 2.c.1]{LT2}.
 Therefore, the isomorphism is given by the formula
$$\sum_{m\geq0}\alpha_mh_m\mapsto\sum_{m\geq0}\alpha_mz_m$$
satisfies all desired conditions.
\end{proof}

Let $r_k(t)$, $k=0,1,2,\dots$, $t\in [0,1]$, be the Rademacher functions. Then, for all positive integers $n\geq k$ and $j=0,1,\dots,2^k-1$, we set
$r_{n,k,j}:=r_n\chi_{(\frac{j}{2^k},\frac{j+1}{2^k})}$.

\begin{lem}\label{Rad}
Let $E:=E(0,1)$ be a symmetric space with the fundamental function $\phi_E$. Then, for every $k\in\mathbb{N}$ and all $j=0,1,\dots,2^k-1$, we have
\begin{align}\label{l2.7}
\left\|
\sum_{n=k}^\infty a_nr_{n,k,j}
\right\|_E\ge \frac{1}{32}\phi_E(2^{-k})
\left\|
(a_n)
\right\|_{\ell_2}.
\end{align}
\end{lem}
\begin{proof}
Denoting
$$
B:=\left\{
t\in [0,1]:\,\Big|\sum_{n=k}^\infty a_nr_{n,k,j}(t)\Big|\ge\frac12\left\|(a_n)\right \|_{\ell_2}
\right\},$$
by the Paley--Zygmund inequality (see e.g. \cite[p.~8]{Kah}), we have
\begin{align}\label{MB}
m(B)&= m\{t\in [0,1]:\,\Big|\sum_{n=k}^\infty a_nr_{n,k,j}(t)\Big|\ge\frac12\left\|(a_n)\right\|_{\ell_2}\}\nonumber\\
&= 2^{-k}m\{t\in [0,1]:\,\Big|\sum_{n=k}^\infty a_nr_{n}(t)\Big|\ge\frac12\left\|(a_n)\right\|_{\ell_2}\}
\ge\frac{1}{16}\cdot 2^{-k}.
\end{align}
 Hence, from the quasi-concavity of $\phi_E$ it follows
\begin{align*}
\left\|\sum_{n=k}^\infty a_nr_{n,k,j}\right\|_E&\ge \left\| \frac12 \left\|(a_n)\right\|_{\ell_2} \chi_B \right\|_E
=\frac12 \left\|(a_n)\right\|_{\ell_2} \left\|\chi_B \right\|_E\\
&\stackrel{\eqref{MB}}{\ge} \frac12 \left\|(a_n)\right\|_{\ell_2}\phi\big(\frac{2^{-k}}{16}\big)\ge \frac{1}{32}\left\|(a_n)\right \|_{\ell_2} \phi_E(2^{-k}) .
\end{align*}
\end{proof}

\subsection{$p$-convexity, $q$-concavity and close notions}
\label{convexity}
Recall that a symmetric  (sequence/function/operator) space $E$ is said to be  {\it $p$-convex} (resp. {\it $q$-concave})~\cite{LT1,LT2,DDS2014,ArLin}
if there exists a constant $K>0$ such that, for every $x_1,\cdots,x_n$ in $E$, we have
$$ \left\| \left(\sum_{k=1}^n  |x_k|^p  \right)^{1/p}\right\|_E   \le K \left( \sum_{k=1}^n \left\| x_k \right\|_E ^p   \right)^{1/p}$$
(resp.
$$   \left( \sum_{k=1}^n \left\| x_k \right\|_E ^q   \right)^{1/q}\le K\left\| \left(\sum_{k=1}^n  |x_k|^q  \right)^{1/q}\right\|_E).$$

Introduce also the following weaker notions. A symmetric (sequence/function/operator) space $E$ is said to satisfy an {\it upper $p$-estimate} (resp. a {\it lower $q$-estimate}) \cite{KMP, LT1, LT2, DDS2014} if there exists a constant $K>0$ such that, for every $x_1,\cdots,x_n$ of pairwise  left disjointly supported elements in $E$,
$$ \left\| \sum_{k=1}^n x_k \right\|_E   \le K \left( \sum_{k=1}^n \left\| x_k \right\|_E ^p   \right)^{1/p}$$
(resp.
$$   \left( \sum_{k=1}^n \left\| x_k \right\|_E ^q   \right)^{1/q}\le K\left\| \sum_{k=1}^n x_k \right\|_E).$$
Clearly, if a symmetric  (sequence/function/operator) space $E$ is $p$-convex (resp. $q$-concave), then it admits an upper $p$-estimate (resp. a lower $q$-estimate). Conversely, if a symmetric function/sequence space $E$ satisfies an upper $p$-estimate, $p>1$ (resp. a lower $q$-estimate, $q<\infty$), then $E$ is $p_1$-convex for each $p_1\in (1,p)$ (resp. $q_1$-concave for each $q_1\in (q,\infty)$) \cite[Theorem~1.f.7]{LT2}.

According to \cite{MSS}, we will refer to an Orlicz function $M$ as {\it $p$-convex} (resp. {\it $q$-concave}) if the function $t\to M(t^{1/p})$ (resp. $t\to M(t^{1/q})$) is convex (resp. concave) on $(0,\infty)$. By \cite[Lemma~20]{MSS}, $M$ is equivalent to a $p$-convex (resp. $q$-concave) function on $(0,\infty)$  if and only if there exists a constant $C>0$  such that for all $t>0$ and $0<s\le 1$ we have
\begin{equation}\label{eq15a}
M(st)\le Cs^p M(t)
\end{equation}
(resp.
\begin{equation}\label{eq15b}
s^q M(t)\le C M(st)).
\end{equation}
An Orlicz function $M$ is equivalent to a $p$-convex (resp. $q$-concave) function on the interval $[0,1]$  if and only if \eqref{eq15a} (resp. \eqref{eq15b}) holds for all $0<t\le 1$ and $0<s\le 1$ (see also \cite[Lemma 6]{AS14} and \cite[Lemma 11]{JSZ17}). This is equivalent to the $p$-convexity (resp. $q$-concavity) of the sequence Orlicz space $\ell_M$ (see e.g. \cite[pages 121 and 124]{KMP97}). Similarly, an Orlicz function $M$ is equivalent to a $p$-convex (resp. $q$-concave) function on the interval
$[1,\infty)$ if and only if \eqref{eq15a} (resp. \eqref{eq15b}) holds for all $t\ge st\ge 1$, and this is equivalent to the $p$-convexity (resp. $q$-concavity) of the Orlicz space $L_M:=L_M(0,1)$.

It is well known that either of the spaces $L_{p,q}(I)$, where $I=(0,1)$ or $I=(0,\infty)$, and $\ell_{p,q}$, $1<p<\infty$, $1\le q<\infty$, is $q$-convex and admits a lower $p$-estimate in the case when $1\le q\le p<\infty$, and it is $q$-concave and admits an upper $p$-estimate in the case when $1\le p<q<\infty$ (see e.g. \cite{Creekmore,KMP,LT2,LT1,Dilworth}).

It is shown in \cite[Corollary 5.3]{DDS2014} (see also \cite{ArLin}) that if $F$ is a separable symmetric sequence space   satisfying an upper $p$-estimate, $p\in (1,2]$, then the corresponding operator space $C_F$ satisfies an upper $r$-estimate for each $1\le r<p$.
If $F$ is a  $q$-concave separable symmetric sequence space  for some $q\ge 2$, then $C_F$ satisfies a lower $q$-estimate.


Under a slightly stronger assumption, for any $p>0$, we have the  following analogue.
\begin{lemma}\label{noncomest}
Let $F$ be a symmetric (or symmetrically quasi-normed) sequence  space having an upper $p$-estimate for some $p>0$ (or a lower $q$-estimate for some $q>0$).
Then, for every finite sequence $x_1,\cdots,x_n$ of pairwise left and right disjointly supported 
 elements in $C_F$, we have
$$ \left\| \sum_{k=1}^n x_k \right\|_{C_F}   \le D_F \left( \sum_{k=1}^n \left\| x_k \right\|_{C_F} ^p   \right)^{1/p},$$
respectively,
$$   \left( \sum_{k=1}^n \left\| x_k \right\|_{C_F} ^q   \right)^{1/q}\le D_F'\left\| \sum_{k=1}^n x_k \right\|_{C_F}.$$
\end{lemma}
\begin{proof}
Since $x_1,\cdots,x_n$ are  pairwise left and right   disjointly supported, it follows that
\begin{align}\label{dis}\mu\left(\sum_{k=1}^n x_k\right)  = \left(\bigoplus_{k=1}^n \mu(x_k)  \right)^* .
\end{align}
By the assumption that  $F$ satisfies an upper $p$-estimate, we obtain that
$$ \left\| \sum_{k=1}^n x_k \right\|_{C_F} \stackrel{\eqref{dis}}{=}   \left\|\oplus_{k=1}^n \mu(x_k)\right\|_{F} \le D_F \left( \sum_{k=1}^n \left\| \mu(x_k)  \right\|_{F}  ^p   \right)^{1/p} =  D_F\left( \sum_{k=1}^n \left\| x_k \right\|_{C_F} ^p   \right)^{1/p}. $$
The same argument yields the case for lower $q$-estimate.
\end{proof}

The notion of $p$-convexity is closely connected with the important concept of Rademacher type. A Banach space $X$ is said to have {\it Rademacher $q$-type}, where $1\le q\le 2$, if there exists a constant $K> 0$ such that for all $n\in\mathbb{N}$ and $x_j\in X$, $j=1,2,\dots,n$, we have
$$
\int_0^1\left\|\sum_{j=1}^nx_jr_j(t)\right \|_{X}\,dt\le K\Big(\sum_{j=1}^n\left\|x_j\right\|_{X}^q\Big)^{1/q}
$$
Clearly, every Banach space has Rademacher $1$-type. Moreover, if $X$ has Rademacher $q_1$-type and $1\le q_2<q_1\le 2$, then $X$ possesses also Rademacher $q_2$-type.

Assume that a symmetric sequence space $F$ has Rademacher $2$-type. Then, by  \cite[Corollary from Theorem~4]{GT}, the ideal $C_F$ has Rademacher $2$-type as well. This fact combined together with the Kahane-Khintchine inequality (see e.g. \cite[Theorem~II.4]{Kah} or \cite[Theorem~1.e.13]{LT2}) implies that for every $1\le p<\infty$ there exists $K_0\ge 1$ such that for all $n\in\mathbb{N}$ and $x_j\in C_F$, $j=1,2,\dots,n$, we have
\begin{equation}
\label{eq14a}
\Big(\int_0^1\Big\|\sum_{j=1}^nx_jr_j(t)\Big\|_{C_F}^p\,dt\Big)^{1/p}\le K_0\Big(\sum_{j=1}^n\left\|x_j\right\|_{C_F}^2\Big)^{1/2}.
\end{equation}

Let $C_F$ be the ideal of the algebra $B(\cH)$ generated by a symmetric sequence space $F$. We define projections $R_m$ and $P_m$, $m=0,1,2,\cdots$, on $C_F$ by setting: if $x=(x_{i,j})\in C_F$, then
\begin{align*}R_m x = \left\{
    \begin{array}{ll}
           x_{i,j}, &~\max(i,j) \le  m, \\
          0, &~ \rm otherwise, \\
    \end{array}
    \right.
 \end{align*}
and
\begin{align*}P_m x = \left\{
    \begin{array}{ll}
           x_{i,j}, &~\min(i,j) \le  m, \\
          0, &~ \rm otherwise, \\
    \end{array}
    \right.
 \end{align*}
Clearly,
for any $m=0,1,2,\cdots$, 
$\left\|R_m\right\|_{C_F\to C_F}\le 1$ and $\left\|P_m\right\|_{C_F\to C_F} \le 2$ (see e.g.  (2.7) and (2.8) in \cite{AL}). 
Moreover, $(P_{m} -P_n )x$ and $(P_{n} -P_l) x$ are disjointly supported from the left for any upper   triangular operator $x\in B(\cH)$ and  any $l\le n \le m$.

The following result is an extension of \cite[Lemma 4]{AL}.

\begin{proposition}
\label{lemma main} Let $F$ be a separable symmetric sequence space admitting an upper $p$-estimate and having Rademacher $2$-type.
Suppose that $\{u_n\}_{n=1}^\infty\subset F$ is a sequence of upper triangular operators such that $\left\|u_n\right\|_{C_F }\leq M$
and
\begin{equation}\label{eq15}
\left\| \sum_{n=1}^\infty  a_n u_n \right\|_{C_F}\geq M^{-1}\left\| a \right\|_{\ell_2}, \;\;\mbox{for all}~ a=(a_n)_{n=1}^\infty \in \ell_2.
\end{equation}
Then, if $0<\gamma < (MK_0)^{-1}$ ($K_0$ is the constant from inequality \eqref{eq14a}), we can find $m\in\mathbb{N}$ such that
\begin{align}\label{eq16}
\norm{P_m u_n} _{{C_F} }\geq\gamma, \;\;n=1,2,\dots
\end{align}
\end{proposition}
\begin{proof}
Assume by contradiction that     for any $m\in \mathbb{N}$,  there exists $n\in \mathbb{N}$ such that
\begin{align}\label{eq17}
\left\|P_m u_n\right\|_{C_F }< \gamma .
 \end{align}
Without loss of generality,
we may assume that for every $m$ the last inequality holds for infinitely many $n$.
Indeed, suppose that  for a given $m$, there exist only finitely many $n$ (denoted by $m(1)$, $m(2)$, $\cdots$, $m(k)$) such that
$\left\|P_m\cdot u_n\right \|_{{C_F}  }< \gamma$.
Since $P_m\uparrow {\bf 1}$,  where ${\bf 1}$ is the identity,
 and $\left\|u_n\right\|_{C_F}  \ge M^{-1} \ge (MK_0)^{-1} >  \gamma$ for all $n\in\mathbb{N}$ (see \eqref{eq15}), it follows from the separability and \cite[Theorem 3.1]{DPS2016} (see also \cite{CS,HLS}) that we can find $m'$ such that
$$
\left\|P_{m'} u_n\right\|_{C_F }>\gamma \;\;\mbox{if}~n=m(1),m(2) ,  \cdots , m(k). $$
Since
$$
\left\|P_{m'} u_n\right\|_{C_F}\ge \left\|
P_m\cdot u_n
\right\|_{C_F }>\gamma,~n\ne m(1) , ~m(2) ,  \cdots , m(k), $$
it follows that   \eqref{eq16} holds with $m'$ instead of $m$. Thus, next we may assume that for every given $m$ inequality \eqref{eq17} holds for infinitely many values of $n\in\mathbb{N}$. Consequently, there are two increasing sequences of positive integers $\{m_i\}_{i=0}^\infty$ ($m_0=0$) and $\{n_i\}_{i=1}^\infty$ such that
\begin{equation}\label{eq18}
\left\|({\bf 1} -R_m)u_{n_i}\right\|_{C_F }< 2^{-i}\;\;\mbox{and}\;\;\left\|P_{m_{i-1}}u_{n_i}\right\|_{C_F }< \gamma,\;\;i=1,2,\dots
 \end{equation}

Next, by  \eqref{eq15} and the triangle inequality for the $L_p$-norm, for each $k\in\mathbb{N}$,  we have
\begin{eqnarray}
M^{-1}k^{1/2}
&\stackrel{\eqref{eq15}} {\le}&
\Big(\int_0^1\Big\|\sum_{i=1}^k r_i(t)u_{n_i}\Big\|_{C_F}^p\,dt\Big)^{1/p}\nonumber\\
&\le&\Big(\int_0^1\Big\|\sum_{i=1}^k r_i(t)R_{m_i}({\bf 1}-P_{m_{i-1}})u_{n_i}\Big\|_{C_F}^p\,dt\Big)^{1/p}\nonumber \\
& + &
\Big(\int_0^1\Big\|\sum_{i=1}^k r_i(t)\Big[u_{n_i}-R_{m_i}({\bf 1}-P_{m_{i-1}})u_{n_i}\Big]\Big\|_{C_F}^p\,dt\Big)^{1/p}.
\label{eq18new}
\end{eqnarray}
Then, since the elements
$$
r_i(t)R_{m_i}({\bf 1}-P_{m_{i-1}})u_{n_i},\;\;i=1,2,\dots,k,$$
are pairwise left and right disjointly supported in $C_F$,
it follows from  Lemma \ref{noncomest}  that the first term from the right-hand side of inequality \eqref{eq18new} does not exceed the quantity
$$
D_F\Big(\sum_{i=1}^k \|R_{m_i}({\bf 1}-P_{m_{i-1}})u_{n_i}\|_{C_F}^p\Big)^{1/p}\le D_FMk^{1/p},$$
where $D_F$ is the $p$-upper estimate constant of $F$.
Moreover, applying \eqref{eq14a}, we see that the second term in \eqref{eq18new} is not bigger than
$$
K_0\Big\{\sum_{i=1}^k \Big(\left\|({\bf 1}-R_{m_i})u_{n_i}\right\|_{C_F}
+\left\|R_{m_{i}}P_{m_{i-1}}u_{n_i}\right \|_{C_F}\Big)^2\Big\}^{1/2}\le K_0(1+\gamma k^{1/2}).$$
Summing up the last estimates, we conclude that for all $k\in\mathbb{N}$
$$
M^{-1}k^{1/2}\le D_FMk^{1/p}+K_0(1+\gamma k^{1/2}),$$
whence
$$
(M^{-1 }-K_0\gamma)k^{1/2}\le D_FMk^{1/p}+K_0.$$
Since $\gamma < (MK_0)^{-1}$, this contradicts the fact that $p>2$.
\end{proof}

Let $\cF$ be a symmetric ideal of $B(\cH)$.
For $1\le p<\infty$, we  define (see e.g. \cite{Xu,DDS2014}) the $p$-convexification $\cF^{(p)}$ by setting
$$  \cF^{(p)} =\{  x \in B(\cH) : |x|^p \in  \cF \} , \quad \left\| x\right\|_{  \cF^{(p)}} = \left\| |x|^p \right\|^{1/p}_ \cF,$$
and  the $p$-concavification  $\cF_{(p)}$ by setting
$$  \cF_{(p)} =\{  x \in B(\cH) : |x|^{1/p} \in  \cF \} , \quad \left\| x \right\|_{  \cF_{(p)}} = \left\| |x|^{1/p} \right\|^{p}_ \cF. $$
The  Lorentz--Shimogaki Theorem   \cite[Theorem 2 and Lemma 3]{LS} demonstrates that the notions of $p$-convexifications and $p$-concavifications
are closely related to the interpolation theory of operators.
We compare below  the norms of $\bigoplus_j x_j$ and  $\sum _j x_j$ in the ideal $C_F$ generated by an interpolation space $F$ between $\ell_2$ and $\ell_\infty$ by applying the  Lorentz--Shimogaki Theorem.

\begin{lem}\label{rdisjoint estimate} Let $F\in{\rm Int}(\ell_2,\ell_{\infty}).$ 
There exists a constant $C>0$ such that for any  elements $x_j\in C_F $, $1\le j\le n$,  which are disjointly supported from the right (or from the left), we have
$$\left\|\sum_{j=1}^n x_j \right\|_{C_F}\geq C \left\|\bigoplus_{j=1}^n   x_j\right\|_{C_F}.$$
\end{lem}
\begin{proof} Let $x_j=u_j|x_j|$, $1\le j\le n$, be the  polar decompositions.
Due to the assumption that $x_j$ have disjoint right supports, we have $\left|x_j\right|\left|x_k\right|=0$ for $j\neq k.$ Thus,
$$\left|
\left(
\sum_{j=1}^n  x_j
\right)^*
\right|^2=\sum_{1\le j,k\le n }  x_jx_k^*
=
\sum_{1\le j,k\le n }  u_j|x_j||x_k|u_k^*
=
\sum_{j=1}^n u_j|x_j|^2u_j^*.$$
By   \cite[Lemma 3.3.7]{LSZ}, we have that
$$\bigoplus_ {j=1}^n  u_j|x_j|^2u_j^*\prec\prec\sum_{j=1}^n  u_j|x_j|^2u_j^*.$$
By the Lorentz--Shimogaki Theorem  (see e.g.  \cite[Theorem 3.1]{Cadilhac}, \cite[Theorem 2 and Lemma 3]{LS} and  \cite{Sparr}),
the (quasi-)norm in $(C_F)_{(2)}$ is monotone with respect to  submajorisation (up to a constant).   
It follows that
\begin{align*}
\left\|\sum_{j=1}^n x_j \right\|_{C_F}
&=\left\|\left|
\left(
\sum _{j=1}^n    x_j
\right)^*
\right|^2 \right\|_{(C_F)_{(2)}}^{\frac12}
= \left\|\sum_{j=1}^n u_j|x_j|^2u_j^*\right\|_{(C_F)_{(2)}}^{\frac12} \\
&\geq C \left\|\bigoplus _{j=1}^n u_j|x_j|^2u_j^*\right\|_{(C_F)_{(2)}}^{\frac12} = C \left\|\bigoplus_{j=1}^n  |x_j|^2 \right\|_{(C_F)_{(2)}} ^{\frac12}=
C \left\|\bigoplus_{j=1}^n    x_j \right\|_{C_F}.
\end{align*}
\end{proof}


\section{Symmetric spaces located between $L_1$ and $L_2$}
\label{<2}

We start with considering the simpler situation. It is well known that the $L_p$-spaces, for $1\le p<2$, have a much richer geometric structure than in the case when $2<p<\infty$. The same observation is true also for symmetric spaces located between $L_1(0,1)$ and $L_2(0,1)$, comparing with spaces lying between $L_2(0,1)$ and $L_\infty(0,1)$. This observation combined with the above-mentioned Arazy's result (see \cite[Corollary 3.2]{Arazy81}) allows rather simply to show that $E\not\hookrightarrow C_F$ for wide classes of symmetric function spaces $E$ and symmetric sequence spaces $F$.

\begin{proposition}
\label{prop easy}
Let $E:=E(0,1)$ be a symmetric function space and $F$ be a separable symmetric sequence space.
Then,

(i) if  $E$ contains a symmetric sequence space $G$ such that  $G\not\hookrightarrow \ell_2\oplus F$, then $E\not\hookrightarrow C_F$.
In particular, if $E$ contains a subspace isomorphic to $\ell_r$ for some $1\le r<\infty$, $r\ne 2$ and $\ell_r \not\hookrightarrow  F$, then
 $E\not\hookrightarrow C_F$;

(ii) If $t^{-1/r}\in E$ for some $r\in (1,2)$, then $E\not\hookrightarrow C_{p,q}:=C_{\ell_{p,q}}$ for all $1<p<\infty$, $1\le q<\infty$, and $E\not\hookrightarrow C_{\lambda_w ^{q}} $ for any $1\le q<\infty$ and decreasing sequence of positive numbers $w=(w_n)_{n=0}^\infty$ such that $\lim_{n\to\infty}w_n=0$ and $\sum_{n=0}^\infty w_n=\infty$. In particular, $L_{p_1,q_1}\not\hookrightarrow C_{p_2,q_2}$ for all $1<p_1<2$, $1<p_2<\infty$, $1\le q_1,q_2<\infty$, and $\Lambda^{q_1}_\psi(0,1)\not\hookrightarrow C_{\lambda_w ^{q_2}}$ for an arbitrary Lorentz space $\lambda_w ^{q_2}$ if $\int_0^1 t^{-q_1/r}d\psi(t) <\infty$ with some $1<r <2$;

(iii) If $p_1,p_2\in (1,\infty)$ and $q_1,q_2\in [1,\infty )$ such that  $q_1\ne q_2$, $q_1\ne 2$, then
$L_{p_1,q_1}(0,1)\not\hookrightarrow C_{p_2,q_2}$ and  $\Lambda_{\psi}^{q_1}(0,1)\not\hookrightarrow C_{\lambda_w^{q_2}}$ for every increasing concave function $\psi$, $\psi(0)=0$, and any decreasing sequence of positive numbers $w=(w_n)_{n=0}^\infty$ satisfying the same properties as in (ii).
\end{proposition}
\begin{proof}
Assertion (i) is an immediate consequence of \cite[Corollary~3.2]{Arazy81}.

(ii). Fix $r\in (1,2)$, $r\ne q$, such that $t^{-1/r}\in E$. Then, by \cite[Theorem~2.f.4]{LT2}, $\ell_r\hookrightarrow E$. According to (i), it remains to show that $\ell_r \not\hookrightarrow  \ell_{p,q}$ and $\ell_r \not\hookrightarrow  \lambda_w^q$. Since the proof of these relations is quite similar, we check only the first of them.

Assuming the contrary, we have a sequence $\{x_k\}\subset \ell_{p,q}$, which  is equivalent in $\ell_{p,q}$ to the unit vector basis in $\ell_r$ and hence is weakly null in $\ell_{p,q}$. Therefore, applying the Bessaga--Pelczy\'{n}ski selection principle (see e.g. \cite[Proposition 1.a.12]{LT1} or \cite[Proposition~1.3.10]{AK}), we can assume that $\{x_k\}$ consists of pairwise disjoint elements. Then, by \cite[Theorem~5]{Dilworth}, the span $[x_k]$ contains a further subspace isomorphic to $\ell_q$. Since the spaces $\ell_q$ and $\ell_r$, $q\ne r$, are totally incomparable (see e.g. \cite[Corollary~2.1.6]{AK}),  this is a contradiction. Hence, the first assertion of (ii) follows.

The second assertion is a direct consequence of the first one, because $t^{-1/r}\in L_{p,q}(0,1)$ if  $p<r<2$ (resp. $t^{-1/r}\in \Lambda^q_\psi(0,1)$ if $\int_0^1 t^{-q/r }d\psi(t) <\infty$).

(iii). The desired result follows from the fact that $\ell_{q_2} \hookrightarrow  L_{p_2,q_2}(0,1)$ \cite[Theorem 11]{Dilworth} and the part (i) in the same way as (ii).
\end{proof}

\section{{Symmetric spaces located between $L_2$ and $L_\infty$}}
\label{first}

In this section we prove the main results of the paper.
Key roles will be played by the following propositions.
Extending Theorem 6 in \cite{AL}, we provide conditions, under which a symmetric function space $E(0,1)$ fails to be isomorphically embedded into the ideal $C_F$ generated by a symmetric sequence space $F$.
Examples of the spaces $E$ and $F$ which satisfy all the above conditions will be given in Section \ref{app}.

\begin{proposition}\label{lem:4}
Let a separable symmetric sequence space $F$ is  $q$-concave and has  an upper $p$-estimate for some $2<p\le q<\infty$. Assume also that
a separable symmetric function space $E$ on $[0,1]$ is such that $\delta_{\phi_E}<1/2$ and
\begin{align}\label{EF}
\phi_E(1/n)\ge A  n^{- 1/ q} ,\;\;n\in\mathbb{N}.
\end{align}
Then, we have
$$E\not\hookrightarrow  C_F.$$
\end{proposition}

A disadvantage of Proposition~\ref{lem:4}
consists in the fact  that a function space $E(0,1)$ must satisfy rather restrictive conditions. The next proposition provides an alternative criterion for $E(0,1)$ ensuring the lack of isomorphic  embeddings into the ideal $C_F$. We relax the restriction imposed on  $E(0,1)$
in Proposition~\ref{lem:4},
and ask for an additional condition of distributional concavity on the  sequence space $F$.
However, since all Orlicz spaces and weighted Lorentz spaces are distributionally concave \cite{LS}, Proposition \ref{prp main}
has a much wider applicability.


\begin{proposition}
\label{prp main}
Let a separable symmetric function space $E:=E(0,1)$ and a separable symmetric sequence space $F$ satisfy the conditions:

(a) $F$ admits an upper $p$-estimate for some $p>2$;

(b) $F$ is distributionally concave;

(c) $F$ 
 is $q$-concave for some $q<\infty$;

(d) there is $A>0$ such that
$$
\left\|\sigma_{1/n}\right\|_{F\to F}\le A \phi_E(1/n) ,\;\;n\in\mathbb{N};$$

(e)
$$
\lim_{t\to\infty}\frac{{\mathcal M}_{\phi_E}(t)}{t^{1/2}}=0.$$

Then, $E\not\hookrightarrow C_F$.
\end{proposition}

\begin{rem}
\label{rem1}
The  fact that $F$ is $q$-concave and has an upper $p$-estimate for $2<p\le q<\infty$ together the Boyd interpolation theorem (see e.g. \cite[the discussion on p.132 and Theorem 2.b.11]{LT2}, \cite{KM-Handbook} or \cite[Theorem~3.5.16]{Bennett_S}) implies that $F$ is   an interpolation space with respect to the couple $(\ell_2,\ell_q)$. Moreover, by the same reasons, $F$ has  the Rademacher type 2 \cite[Proposition~1.f.3 and Theorem~1.f.7]{LT2}. 
These observations allow us to use Proposition \ref{lemma main} and Lemmas \ref{triangular lemma}, \ref{right disjoint} in the proof below.
\end{rem}


\subsection{Common part of two proofs}

On the contrary, suppose that there exists an isomorphic embedding $T$ of $E$ into $C_F$.
Applying Lemmas \ref{triangular lemma} and \ref{right disjoint}, we may assume that $T$ maps $E$ into the upper triangular part $U_{F}$ of $C_F$ and that, for each $n\geq0,$ the elements $T(h_{2^n+i}),$ $0\leq i<2^n,$ where $h_j$, $j=1,2,\dots$, are the Haar functions, are disjointly supported from the right.

From the definitions of the Rademacher and Haar functions it follows
$$
r_k(t) = \sum _{i=0}^{2^k-1} h_{2^k +i  }, ~t\in [0,1], ~k=0,1,2,\dots$$
Hence, recalling that, for any positive integers $n\geq k$ and $j=0,1,\dots,2^k-1$, $r_{n,k,j}:=r_n\chi_{(\frac{j}{2^k},\frac{j+1}{2^k})}$ (see Section~\ref{Haar, Rademacher}), we have
$$
r_{n,k,j}=\sum_{i=0}^{2^{n-k}-1}h_{2^n+j2^{n-k}+i}.$$
Therefore, for any fixed positive integers $k$ and $n\geq k$, the elements $T(\widehat{r_{n,k,j}}),$ $j=0,1,\dots,2^k-1,$ where
$$
\widehat{r_{n,k,j}}:=\frac{r_{n,k,j}}{\phi_E(2^{-k})},$$
are also disjointly supported from the right.

Without loss of generality, we will suppose that $\|T^{-1}\|=1$. Then, by Lemma~\ref{Rad},
\begin{eqnarray}
\left\|
\sum_{n=k}^\infty a_nT(\widehat{r_{n,k,j}})\right\|_{C_F}&=&\left\|
T\Big(\sum_{n=k}^\infty a_n\widehat{r_{n,k,j}}\Big)\right\|_{C_F}\nonumber\\
&\ge& \left\|\sum_{n=k}^\infty a_n\widehat{r_{n,k,j}}\right\|_{E} \stackrel{\eqref{l2.7}}{\ge} \frac{1}{32}\left\|(a_n) \right \|_{\ell_2}.
\label{eq10}
\end{eqnarray}
Applying now Proposition \ref{lemma main} to the sequence $\{u_n\}_{n=1}^\infty$,  $u_n:=T(r_n)=T(\widehat{r_{n,0,0}})$, we find $m_1\in\mathbb{N}$ such that
\begin{equation}
\label{eq11}
\left\|P_{m_1}T(r_n)\right\|_{C_F}\ge \gamma,\;\;n=1,2,\dots,
\end{equation}
where $\gamma:=\left(2K_0\max(32,\|T\|)\right)^{-1}$ ($K_0$ is the constant from inequality \eqref{eq14a}).

Let $n,k\in\mathbb{N}$, $n\ge k$,
be fixed.
Denote by $Q_{n,k}$ the natural isometry of the span $E_k:=[x_{k,j},j=0,1,\dots,2^k-1]$ in $E$,
where $x_{k,j}:=\phi_E(2^{-k})^{-1}\chi_{\Delta_k^j}$ ($\Delta_k^j=[j2^{-k},(j+1)2^{-k})$, $k=1,2,\dots$, $j=0,1,\dots,2^k-1$),
onto the span $Y_k:=[\widehat{r_{n,k,j}},j=0,1,\dots,2^k-1]$ in $E$.
Observe that, for each $m\in\mathbb{N}$,
the space
$(P_mU_F,\left\|\cdot\right\|_{C_F})$ is isomorphic to $\ell_2$, and hence the operator $P_{m_1}TQ_{n,k}$
is bounded from $E_k$ in $\ell_2$. By the condition $\delta_{\phi_E}<1/2$ (see also Remark \ref{rem lem dim}), we can use Lemma \ref{finite dim}. Consequently, for every $\epsilon>0$ there is $k_1\in\mathbb{N}$ such that for all $n\ge k_1$
$$
\Big|\Big\{j=0,1,\dots,2^{k_1}-1:\,
\frac{
\left\|
P_{m_1}T(\widehat{r_{n,k_1,j}})
\right\|_{\ell_2}}{
\left\|
P_{m_1}T \right\|_{E_k\to \ell_2}}\ge\epsilon\Big\}\Big|\le\epsilon\cdot2^{k_1},$$
 where for a finite set $B$ we put $|B|:={\rm card}\,B$. Noting that
$$  \frac{\left\|P_{m_1}T(\widehat{r_{n,k_1,j}})
\right\|_{C_F}}
{\left\|
P_{m_1}T(\widehat{r_{n,k_1,j}})
\right\|_{\ell_2}}
  \le   d(P_{m_1}U_F,\ell_2) $$
and
$$\left\|
P_{m_1}T \right\|_{E_k\to \ell_2}
\le \left\|
P_{m_1}T \right\|_{E\to \ell_2} \le \left\|
P_{m_1}T \right\|_{E\to C_F }   d(P_{m_1}U_F,\ell_2), $$
we conclude that
$$\frac{\left\|P_{m_1}T(\widehat{r_{n,k_1,j}})
\right\|_{C_F}}
{C(m_1,T)}
  \le \frac{
\left\|
P_{m_1}T(\widehat{r_{n,k_1,j}})
\right\|_{\ell_2}}{
\left\|
P_{m_1}T \right\|_{E_k\to \ell_2}},$$
where
$$
C(m_1,T):=\left\|P_{m_1}T\right \|_{E\to C_F}\cdot d(P_{m_1}U_F,\ell_2)^2$$ ($d(X,Y)$ is the Banach--Mazur distance between Banach spaces $X$ and $Y$),
we obtain
$$
\left|
\{j=0,1,\dots,2^{k_1}-1:\,
\left\|
P_{m_1}T(\widehat{r_{n,k_1,j}})
\right\|_{C_F}\ge\epsilon C(m_1,T)\}
\right|
\le
\epsilon\cdot2^{k_1}.$$
Choose $\epsilon\in (0,1/2)$ so that $\gamma>2\epsilon C(m_1,T)$. Then, from the preceding inequality it follows that
$$
|\{j=0,1,\dots,2^{k_1}-1:\, \left\|P_{m_1}T(\widehat{r_{n,k_1,j}})\right\|_{C_F}\ge\gamma/2\}|\le 2^{k_1-1}.$$
Therefore, letting
$$
A_n:=\{j=0,1,\dots,2^{k_1}-1:\, \left\|P_{m_1}T(\widehat{r_{n,k_1,j}})\right\|_{C_F}\le\gamma/2\},$$
for all $n\ge k_1$ we get
\begin{equation}
\label{eq12}
|A_n|\ge 2^{k_1-1}.
\end{equation}

Next, in view of inequality \eqref{eq10} and Remark \ref{rem1}, we can apply Proposition \ref{lemma main} to each of the sequences $(T(\widehat{r_{n,k_1,j}}))_{n\ge k_1}$, $j=0,1,\dots,2^{k_1}-1$, and find a positive integer $m_2>m_1$ such that
$$
\left\|P_{m_2}T(\widehat{r_{n,k_1,j}})\right\|_{C_F}\ge\gamma,\;\;n\ge k_1,\;j=0,1,\dots,2^{k_1}-1.$$
Combining this inequality together with the definition of sets $A_n$, we infer
\begin{equation}
\label{eq13}
\left\|(P_{m_2}-P_{m_1})T(\widehat{r_{n,k_1,j}})\right \|_{C_F}\ge\gamma/2,\;\;n\ge k_1,\;j\in A_n.
\end{equation}

As was mentioned above, the elements $T(\widehat{r_{n,k,j}}),$ $0\leq j<2^k,$ are disjointly supported from the right. Clearly, the elements
$$
(P_{m_2}-P_{m_1})T(\widehat{r_{n,k_1,j}}),\;\;0\leq j<2^{k_1},$$
have the same property.

\subsection{Proof of Proposition \ref{lem:4}}


\begin{proof}[Proof of Proposition \ref{lem:4}]
Observe that $$
T(r_n)=\phi_E(2^{-k_1})\cdot \sum_{j=0}^{2^{k_1}-1}T(\widehat{r_{n,k_1,j}}). $$
By the lower $q$-estimate of $C_F$ \cite{ArLin} (see also \cite[Proposition 5.1]{DDS2014}), we have
\begin{align*}
\norm{(P_{m_2}- P_{m_1}) T(r_n)}_{C_{F}} \ge c_F
\phi_E(2^{-k_1})  \left(
 \sum_{j=0}^{2^{k_1 } -1}
\norm{(P_{m_2}- P_{m_1}) T(\widehat{r_{n,k_1,j}}) }_{C_F}^{q }
\right)^{1/q},
\end{align*}
where $c_F$ is the lower $q$-estimate constant for $C_F$. Moreover, thanks to condition~\eqref{EF},
$$ \phi_E(2^{-k_1}) \cdot   2^{ k_1/q}\ge A.$$
Thus, from \eqref{eq12} and \eqref{eq13} it follows that for $n\ge k_1$
\begin{align*}
\norm{(P_{m_2}- P_{m_1}) T(r_n)}_{C_{F}} \ge c_F \cdot \phi_E(2^{-k_1}) \cdot  2^{ \frac{k_1 -1}{ q}  }  \cdot   \frac{\gamma }{2} \ge c_F\cdot  A \cdot 2^{-1/ q} \frac{\gamma}{2 } \ge \frac{A}{4} c_F \gamma .
\end{align*}

Next, using Lemma \ref{finite dim} once more and repeating the same reasoning, we find a positive integer $k_2>k_1$ such that
$$
|\{j=0,1,\dots,2^{k_2}-1:\, \left\|P_{m_2}T(\widehat{r_{n,k_2,j}})\right\|_{C_F}\ge\gamma/2\}|\le 2^{k_2-1}.$$
Applying then Proposition \ref{lemma main} to each of the sequences $(T(\widehat{r_{n,k_2,j}}))_{n\ge k_2}$, $j=0,1,\dots,2^{k_2}-1$,
and
repeating the above argument, we find  $m_3>m_2$ satisfying the condition
$$
\left\|(P_{m_3}-P_{m_2})T(r_n )\right \|_{C_F}\ge \frac{A}{4}c_F  \gamma \;\;\mbox{for all}\;\; n\ge k_2.$$
Proceeding in the same way, we obtain two increasing sequences of positive integers $\{k_i\}_{i=1}^\infty$ and $\{m_i\}_{i=1}^\infty$ such that for all $n\ge k_i$
\begin{equation}
\label{eq13ab}
\left\|(P_{m_{i+1}}-P_{m_i})T(r_n)\right\|_{C_F}\ge \frac{A}{4} c_F \gamma,\;\;i=1,2,\dots
\end{equation}

Fix $l\in\mathbb{N}.$ For each $n\geq k_l,$ the elements $ (P_{m_{i+1}}-P_{m_i})T(r_n),$ $1\leq i \leq l-1,$ are disjointly supported from the left
with   left supports $ \sum_{j= m_i+1}^{m_{i+1}}e_{jj} $.
 Hence, by the lower $q$-estimate of $F$ \cite{ArLin} (see also \cite[Proposition 5.1]{DDS2014}), we have
$$
\left\|T(r_n)\right\|_{C_F }\geq
c_F
\left (
\sum_{i=1}^{l-1} \left\|(P_{m_{i+1}}-P_{m_{i}})T(r_n)\right\|_{C_F }^{q}
\right)^{1/q}
\stackrel{\eqref{eq13ab}}{\geq}
 c_F (l-1)^{1/q} \frac{Ac_F\gamma }{4}.$$
Since $l\in\mathbb{N}$ is arbitrary, this implies that the operator $T$ is unbounded, which contradicts the hypothesis.
\end{proof}


\subsection{Proof of Proposition \ref{prp main}}
\begin{proof}[Proof of Proposition \ref{prp main}]
Taking into account that
$$
T(r_n)=\phi_E(2^{-k_1})\cdot \sum_{j=0}^{2^k-1}T(\widehat{r_{n,k_1,j}})$$
and  $(P_{m_2}-P_{m_1})T(\widehat{r_{n,k_1,j}})$, $0\le j<2^{k_1}$,  are disjointly supported from the left,
by Lemma \ref{rdisjoint estimate}, we have
\begin{align}\label{directsum}
\left\|(P_{m_2}-P_{m_1})T(r_n) \right \|_{C_F}
&\geq \phi_E(2^{-k_1})\left\|\bigoplus_{j=0}^{2^{k_1}-1}\big((P_{m_2}-P_{m_1})T(\widehat{r_{n,k_1,j}})\big)\right\|_{C_F}\nonumber\\
&\geq
\phi_E(2^{-k_1})\left\|\bigoplus_{j\in A_{n}}\big((P_{m_2}-P_{m_1})T(\widehat{r_{n,k_1,j}})\big)
\right\|_{C_F }.
\end{align}
Here, $\oplus_j \big((P_{m_2}-P_{m_1})T(\widehat{r_{n,k_1,j}})\big)\in \oplus_j \cH  $ stands for the direct sum of operators $(P_{m_2}-P_{m_1})T(\widehat{r_{n,k_1,j}})$ and, under the natural isomorphism from $\oplus_j \cH$  to $\cH$ (see Remark \ref{ncdc}), summands in the direct sum  can be viewed as elements in $\cH$, which are pairwise disjontly supported from the left and from the right.
Since $F$ is distributionally concave, say, with the constant $c_F$, in view of the preceding relations, definition of the dilation operator
and
 the quasi-concavity of the fundamental function $\phi_E$, for all $n\ge k_1$ we obtain

\begin{align*}
\left\|(P_{m_2}-P_{m_1})T(r_n)\right\|_{C_F}
\stackrel{\eqref{directsum}}{\ge} & \phi_E(2^{-k_1})\left\|\bigoplus_{j\in A_{n}}\big((P_{m_2}-P_{m_1})T(\widehat{r_{n,k_1,j}})\big)
\right\|_{C_F }.\\
 \stackrel{\eqref{ncdceq}}{\ge}&
c_F\phi_E(2^{-k_1})\cdot\min_{j\in A_n}
\left\|
\big((P_{m_2}-P_{m_1})T(\widehat{r_{n,k_1,j}})\big)^{\oplus |A_n|}
\right\|_{C_F}\\
 \stackrel{\eqref{eq12}}{\ge} & c_F\phi_E(2^{-k_1})\cdot\min_{j\in A_n}
\left\|
\big((P_{m_2}-P_{m_1})T(\widehat{r_{n,k_1,j}})\big)^{\oplus 2^{k_1-1}}
\right\|_{C_F}\\
{\ge}~&
\frac{c_F\phi_E(2^{-k_1})}{ \left\|\sigma_{2^{1-k_1}}\right\|_{F\to F}}\cdot
\min_{j\in A_n} \left\|(P_{m_2}-P_{m_1})T(\widehat{r_{n,k_1,j }})\right\|_{C_F}\\
 \stackrel{(d) }{\ge}~&
\frac{c_F}
{A}\cdot\frac{\phi_E(2^{-k_1})}{\phi_E(2^{1-k_1})}
\cdot \min_{j\in A_n}
\left\|
(P_{m_2}-P_{m_1})T(\widehat{r_{n,k_1,j}})
\right\|_{C_F}\\
 \ge~&
\frac{c_F}{2A}\cdot  \min_{j\in A_n} \left\|(P_{m_2}-P_{m_1})T(\widehat{r_{n,k_1,j}})\right\|_{C_F}.
\end{align*}
Thus, by \eqref{eq13}, for all $n\ge k_1$ we have
$$
\left\|
(P_{m_2}-P_{m_1})T(r_n)
\right\|_{C_F}\ge c_0\gamma,$$
where $c_0$ depends only on $E$ and $F$.

Next, using Lemma \ref{finite dim} once more and repeating the same reasoning, we find a positive integer $k_2>k_1$ such that
$$
|\{j=0,1,\dots,2^{k_2}-1:\,
\left\|
P_{m_2}T(\widehat{r_{n,k_2,j}})
\right\|_{C_F}\ge \gamma/2\}|\le 2^{k_2-1}.$$
Applying then Proposition~\ref{lemma main} to each of the sequences $(T(\widehat{r_{n,k_2,j}}))_{n\ge k_2}$, $j=0,1,\dots,2^{k_2}-1$, we find  $m_3>m_2$ satisfying the condition
$$
\left\|
(P_{m_3}-P_{m_2})T(r_n)
\right\|_{C_F}\ge c_0\gamma\;\;\mbox{for all}\;\; n\ge k_2.$$
Proceeding in the same way, we obtain two increasing sequences of positive integers $\{k_i\}_{i=1}^\infty$ and $\{m_i\}_{i=1}^\infty$ such that for all $n\ge k_i$
\begin{equation}
\label{eq13a}
\left\|
(P_{m_{i+1}}-P_{m_i})T(r_n)
\right\|_{C_F}\ge c_0\gamma,\;\;i=1,2,\dots \end{equation}

Fix $l\in\mathbb{N}.$ For each $n\geq k_l,$ the elements $(P_{m_{i+1}}-P_{m_i}) T(r_n),$ $1\leq i \leq l-1,$ are disjointly supported from the left. Hence, it follows from Lemma \ref{rdisjoint estimate} that
$$
\left\|T(r_n)\right\|_{C_F }\geq
\left\|\sum_{i=1}^{l-1}(P_{m_{i+1}}-P_{m_{i}})T(r_n)\right\|_{C_F }
\geq
\left\|\bigoplus_{i=1}^{l-1}(P_{m_{i+1}}-P_{m_{i}})T(r_n) \right\|_{C_F}.$$
Moreover, 
since the space $F$ is $q$-concave,
it follows that $F$ satisfies a lower $q$-estimate with a constant $K_F$\cite[p.85]{LT2}.
Thus, applying Lemma \ref{noncomest},
 we obtain that   for every positive integer $l$ and all $n\ge k_l$,
\begin{align*}
& \quad \|T\|
\ge\left\| T(r_n)\right\|_{C_F}\ge \left\|\bigoplus_{i=1}^{l-1}(P_{m_{i+1}}-P_{m_{i}})T(r_n) \right\|_{C_F}\\
&~\ge K_F \big(\left\|P_{m_1}T(r_n)\right\|_{C_F}^q+ \left\|(P_{m_{2}}-P_{m_1})T(r_n)\right\|_{C_F}^q+\dots +\left\|(P_{m_{l}}-P_{m_{l-1}})T(r_n)\right\|_{C_F}^q\big)^{1/q}\\
&\stackrel{ \eqref{eq13a}} {\ge}K_F  c_0\gamma (l-1)^{1/q}.
\end{align*}
Since $l\in\mathbb{N}$ is arbitrary, this implies that the operator $T$ is unbounded, which contradicts the hypothesis.
\end{proof}

\section{Applications}\label{app}
The most natural examples  satisfying all the conditions of Propositions \ref{lem:4} and \ref{prp main} are $L_{p,q}$-spaces, with $2<  q\le p<\infty$ and suitable Orlicz spaces.

Recall that the Lorentz sequence space $\ell_{r,s}$, $1<r<\infty$, $1\le s<\infty$, admits an upper $\min(r,s)$-estimate and a lower $\max(r,s)$-estimate (see Section~\ref{convexity} and \cite{Creekmore,KMP,LT2,LT1,Dilworth}). Therefore, $\ell_{r,s}$ has a lower $q$-estimate and an upper $p$-estimate for some $2<p\le q<\infty$ if and only if $2<r,s<\infty$.
Moreover, since $\phi_{L_{r_1,s_1}}(t)=t^{1/r_1}$, $1<r_1<\infty$, $1\le s_1<\infty$, one can easily verify that the rest of conditions of Proposition \ref{lem:4} for the spaces $F=\ell_{r,s}$ and $E=L_{r_1,s_1}$ is satisfied if and only if $r_1\ge \max(r,s)$. From these observations and Proposition \ref{lem:4} we get the following result that extends \cite[Theorem 6]{AL} to the class of $L_{p,q}$-spaces. Namely,  \cite[Theorem 6]{AL}  is the special case of the assertion below when $p_1=p_2= q_1=q_2$.
\begin{thm}\label{Lor}
For arbitrary $1\le q_1<\infty$, $q_2,p_2>2$, and $\max(p_2,q_2)\le p_1<\infty$, we have
$$
L_{p_1 ,q_1}(0,1)\not\hookrightarrow  C_{p_2,q_2}. $$
In particular,  if $2<q\le p<\infty  $, then
$$
L_{p,q}(0,1)\not\hookrightarrow C_{p,q}. $$
\end{thm}


We apply Proposition \ref{prp main} below.
\begin{thm}\label{Orl}
Let $M$ and $N$ be Orlicz functions satisfying the following conditions:

(i) $N$ is equivalent to a $p$-convex and $q$-concave Orlicz function on $[0,1]$ for some  $2<p\le q<\infty$;


(ii) there is $A>0$ such that
$$
N(uv)\le A N(u)M(v)\;\;\mbox{if}\;\;0<u\le 1,\; v\ge 1;$$

(iii)
$M$ is equivalent to a $r$-convex Orlicz function on $[1,\infty)$ for some  $r>2$.

Then, $L_{M}[0,1]\not\hookrightarrow  C_{\ell_N}$.
\end{thm}

\begin{proof}
Let us check that conditions (a) --- (e) of Proposition \ref{prp main} are fulfilled for the spaces $E=L_{M}[0,1]$ and $F=\ell_{N}$. Without loss of generality, we assume that $N(1)=M(1)=1$.

First of all, condition (i) implies that the space $F$ is $p$-convex and $q$-concave (and so is separable) \cite{KMP97}.
Thus, since every Orlicz space is distributionally concave (see Section~\ref{Distributional concavity}), the space $F$ satisfies conditions (a) --- (c).

Let us check that (d) is a consequence of conditions (i) and (ii) of the theorem. Indeed, by \cite[Theorem~6]{LS1} (see also \cite[\S\,4, p.~28]{Mal-85} and \cite{Boyd}), we have 
$$
\left\|\sigma_{1/n} \right\|_{F\to F}\le 2\cdot \sup_{m\in\mathbb{N}}\frac{\phi_F(m)}{\phi_F(nm)}.
$$
Therefore, since $\phi_F(m)={1}/{N^{-1}(1/m)}$, $m\in\mathbb{N}$,
it follows
$$
\left\|\sigma_{1/n}\right\|_{F\to F}\le 2\sup_{m\in\mathbb{N}}\frac{N^{-1}(1/(nm))}{N^{-1}(1/m)}\le 2\sup_{0<s\le 1}\frac{N^{-1}(s/n)}{N^{-1}(s)}.$$
On the other hand,
\begin{equation}
\label{eq13b}
\phi_E(t)=\frac{1}{M^{-1}(1/t)},\;\;t>0,
\end{equation}
and so condition (d) is satisfied if there is a constant $C>0$ such that for all $0<s,t\le 1$
\begin{equation}
\label{eq13c}
N^{-1}(st)M^{-1}(1/t)\le CN^{-1}(s).
\end{equation}
Passing to the inverse functions, we get that it is equivalent to the inequality
$$
N(C^{-1}N^{-1}(st)M^{-1}(1/t))\le s,$$
or, after the changes of variables $u=N^{-1}(st)$ and $v=M^{-1}(1/t)$, to the following:
\begin{equation*}
N(C^{-1}uv)\le N(u)M(v),\;\;0<u\le 1,\; v\ge 1.
\end{equation*}
Clearly, we can assume that the constant $A$ in (ii) is bigger than $1$. Therefore, applying successively conditions (ii) and (i), we get
$$
N(A^{-1/p}uv)\le AN(A^{-1/p}u)M(v)\le N(u)M(v),\;\;0<u\le 1,\; v\ge 1.$$
Thus, the preceding inequality and also \eqref{eq13c} hold with $C=A^{1/p}$, which implies (d).

It remains to show that the space $E=L_M[0,1]$ satisfies condition (e). To this end,
we observe that, by definition and \eqref{eq13b}, for every $u\ge 1$
$$
{\mathcal M}_{\phi_E}(u)=\sup_{0<v\le 1/u}\frac{\phi_E(vu)}{\phi_E(v)}=
\sup_{w\ge u}\frac{M^{-1}(w)}{M^{-1}(w/u)}.$$
Therefore, since $r>2$, it suffices to prove that there is a constant $C>0$ such that
\begin{equation}
\label{eq13d}
M^{-1}(w)\le Cu^{1/r}M^{-1}(w/u),\;\;w\ge u\ge 1.
\end{equation}
One can easily check that this is equivalent to the following:
$$
uM(z)\le M(Cu^{1/r}z),\;\;u,z\ge 1.$$
In turn, after the change $s=u^{-1/r}$ we come to the inequality
$$
M(z)\le s^rM(Cz/s),\;\;z\ge 1,\;0<s\le 1,$$
and then, setting $t=z/s$, to
$$
M(st)\le s^rM(Ct),\;\;t\ge st\ge 1.$$
On the other hand, it is easy to see that the latter inequality is an immediate consequence of condition (iii) of the theorem. Indeed, (iii) means (see Subsection~\ref{convexity}) that for some constant $C'\ge 1$
$$
M(st)\le C's^r M(t),\;\;t\ge st\ge 1.$$
Combined this with the fact that $C'M(t)\le M(C't)$, $t>0$, because $M$ is an increasing convex function, we obtain the preceding inequality with $C=C'$. This completes the proof of the theorem.
\end{proof}

As a consequence we get the following result.

\begin{cor}
\label{cor: Orl}
Let $M$ be an Orlicz function satisfying the following conditions:

($\alpha$) $M$
is equivalent to an Orlicz function that is $p$-convex on $[0,\infty) $ and $q$-concave on $[0,1]$ for some  $2<p\le q<\infty$;


($\beta$) there is $A>0$ such that
$$
M(uv)\le A M(u)M(v)\;\;\mbox{if}\;\;0<u\le 1,\; v\ge 1.$$

Then, $L_{M}[0,1]\not\hookrightarrow  C_{\ell_M}$.

In particular, this result holds for every submultiplicative Orlicz function $M$, which is $p$-convex for some $p>2$.
\end{cor}

\begin{proof}
The first assertion of the corollary is a straightforward consequence of Theorem~\ref{Orl}.
 To show the second one, it suffices to check that $M$ is equivalent to a $q$-concave Orlicz function on $[0,1]$ for some finite $q$. To this end, let us observe that the submultiplicativity of $M$ clearly implies that $M\in \Delta_{2}^{0}$ (see Subsection~\ref{deflor}). In turn, then there exist $q<\infty$ and $C>0$ such that for all $0<s,t\le 1$ we have $s^qM(t)\le CM(st)$ (see e.g. \cite[Proposition on p.~121]{KMP97}). This observation finishes the proof (see Subsection~\ref{convexity}).
\end{proof}


%

 \section{Lack of isomorphic embeddings of $L_{2,q}$-spaces into ideals $C_{2,q}$.}
 \label{t2}

In the previous sections, we considered symmetric spaces, which are located either between $L_2$ and $L_\infty$ or between $L_1$ and $L_2$.
In this final part of the paper, we study the spaces $L_{2,q}$, $1\le q<\infty$, which are being the most typical examples of spaces ``very close''\:to the space $L_2$ (in particular, they belong to  neither  the set
${\rm Int}(L_1,L_2)$  nor ${\rm Int}(L_2,L_\infty)$).
We show that $L_{2,q}$ cannot be isomorphically embedded into the ideal $C_{2,q}$ for every $1\le q<\infty$. Here, we make use of the above-mentioned Arazy's result \cite[Corollary~3.2]{Arazy81}, some properties of sequences of independent functions in $L_{2,q}$-spaces obtained in~\cite{CD89,AS08} and recent results on the embeddings of $\ell_{p,q}$-spaces from  \cite{KS,SS2018}.

\begin{theorem}\label{theor:easy}
For every $1\le q<\infty$, $q\ne 2$, the space $L_{2,q}:=L_{2,q}(0,1)$ fails to be isomorphically embedded into the ideal $C_{2,q}$.
\end{theorem}

\begin{proof}
Assume by contradiction that $L_{2,q}$ is isomorphically embedded into $C_{2,q}$. Further, we consider the cases when $1\le q<2$ and $2<q<\infty$ separately.

(a) $1\le q<2$. By Corollary~3.6 from the paper \cite{AS08} (see also \cite{JSZ}),
for every $1\le q<2$ the space $L_{2,q}$ contains a sequence of independent identically distributed mean zero functions $\{f_k\}_{k=1}^\infty$, which is not equivalent in $L_{2,q}$ to the unit vector basis in $\ell_2$.
Moreover, according to \cite[Corollary~3.14]{CD89}, for some $C>0$ we have
\begin{equation}
\label{eq1a}
\left\|\sum_{k=1}^\infty a_kf_k\right\|_{L_{2,q}}\le C\left\|(a_k)\right \|_{\ell_{2,q}}.
\end{equation}
Since $f_k$, $k=1,2,\dots$, are independent and identically distributed, then $\{f_k\}_{k=1}^\infty$ is a symmetric basic sequence in $L_{2,q}$. Therefore, by \cite[Corollary~3.2]{Arazy81}, the closed linear span $[f_k]$ in $L_{2,q}$ is isomorphically embedded into the space $\ell_2\oplus \ell_{2,q}$. This means that there is a sequence $\{x_k\}\subset \ell_2\oplus \ell_{2,q}$ such that for all $a_k\in\mathbb{R}$
\begin{equation}
\label{eq0}
\left\|
\sum_{k=1}^\infty a_kx_k\right\|_{\ell_2\oplus \ell_{2,q}}\asymp \left\|\sum_{k=1}^\infty a_kf_k\right\|_{L_{2,q}}.
\end{equation}
Combining this with the fact that $\{f_k\}$ is not equivalent in $L_{2,q}$ to the unit vector basis in $\ell_2$ and with inequality \eqref{eq1a}, we see that the sequence $\{x_k\}$ is not equivalent in $\ell_2\oplus \ell_{2,q}$ to the latter basis as well and
\begin{equation}
\label{eq1}
\left\|\sum_{k=1}^\infty a_kx_k\right\|_{\ell_2\oplus \ell_{2,q}}\le C_1 \left\| (a_k) \right \|_{\ell_{2,q}}.
\end{equation}

Observe that $f_k$ are independent, $\int_0^1f_k(s)\,ds=0$, $k=1,2,\dots$, and $\left\|f_k \right\|_{L_2}=\left\| f_1\right\|_{L_2}\le \left\|f_1\right\|_{L_{2,q}}$, $k\ge 2$ \cite[Chapter 4, Proposition 4.2]{Bennett_S}. Hence, the functions $f_k/\left\|f_1\right\|_{L_2}$, $k=1,2,\dots$, form an orthonormal system. Therefore, $\int_0^1f_k(s)g(s)\,ds\to 0$ as $k\to\infty$ for each $g\in L_2$. Since $L_2$ is dense in the dual space $(L_{2,q})^*=L_{2,q'}$, $1/q+1/q'=1$, we obtain that $\{f_k\}$ is  weakly null in $L_{2,q}$. Thus, $\{x_k\}$ is weakly null in $\ell_2\oplus \ell_{2,q}$ as well, and hence, applying the Bessaga--Pelczy\'{n}ski selection principle (see e.g. \cite[Proposition 1.a.12]{LT1}), we can assume that $x_k$, $k=1,2,\dots$, are pairwise disjointly supported.

Let $x_k=y_k+z_k$, where $y_k\in \ell_2$ and $z_k\in \ell_{2,q}$, $k=1,2,\dots$. Then,
\begin{equation}
\label{eq2}
\left\|\sum_{k=1}^\infty a_kx_k\right\|_{\ell_2\oplus \ell_{2,q}}\asymp \left\|\sum_{k=1}^\infty a_ky_k\right\|_{\ell_2}+\left\|\sum_{k=1}^\infty a_kz_k\right \|_{\ell_{2,q}}.
\end{equation}
If $\liminf_{k\to\infty} \left\|z_k\right\|_{\ell_{2,q}}=0$, then passing to a subsequence if it is necessary, we get that $\{x_k\}$ is equivalent in $\ell_2\oplus \ell_{2,q}$ to the unit vector basis in $\ell_2$, which contradicts our assumption. Therefore, $\left\|z_k\right\|_{\ell_{2,q}}\ge c$, $k=1,2,\dots$, for some $c>0$.

Let $z_k=(z_k(i))_{i=1}^\infty$. We consider two cases: (i) $\liminf_{k\to\infty} \sup_i|z_k(i)|=0$ and (ii) $\liminf_{k\to\infty} \sup_i|z_k(i)|>0$.

In the case (i), by \cite[Proposition~1]{Dilworth}, passing to a subsequence, we can assume that $\{z_k\}$ is equivalent in $\ell_{2,q}$ to the unit vector basis in $\ell_q$, and then from the inequality $q<2$ and \eqref{eq2} it follows that $\{x_k\}$ is equivalent in $\ell_2\oplus \ell_{2,q}$ to the same basis, which contradicts inequality \eqref{eq1} because $\ell_q\stackrel{\ne}{\subset}\ell_{2,q}$ for $q<2$ \cite[p.217]{Bennett_S}.

In the case (ii), we can find $\delta>0$ such that for each $k=1,2,\dots$ there is a positive integer $i_k$ satisfying the inequality $|z_k(i_k)|\ge \delta$. Then, since $z_k$, $k=1,2,\dots$, are pairwise disjointly supported, it follows that
$$
\left\|\sum_{k=1}^\infty a_kz_k\right\|_{\ell_{2,q}}\ge \delta \left\|(a_k)\right\|_{\ell_{2,q}}.
$$
Hence,  we obtain
$$
\left\|(a_k)\right\|_{\ell_{2,q}}\stackrel{\eqref{eq1}}{\ge} C_1^{-1}\left\|\sum_{k=1}^\infty a_kx_k\right\|_{\ell_2\oplus \ell_{2,q }}\stackrel{\eqref{eq2}}{\ge} c'\left\|\sum_{k=1}^\infty a_kz_k\right\|_{\ell_{2,q}}\ge c\left\|(a_k)\right \|_{\ell_{2,q}},
$$
for some $c>0$. Therefore, taking into account equivalence \eqref{eq0}, we conclude that the sequence $\{f_k\}$ is equivalent in $L_{2,q}(0,1)$ to the unit vector basis in $\ell_{2,q}$, which contradicts the fact that
 $\ell_{2,q}\not\hookrightarrow L_{2,q}(0,1)$, $q\in [1,2)\cup (2,\infty)$ (see \cite{KS,SS2018}). As a result, in the case $1\le q<2$ the proof is completed.

(b) $2<q<\infty$. Our argument will be based on using Proposition~3.11 from the paper \cite{CD89}. According to this result, if $\{g_k\}_{k=1}^\infty$ is a  sequence of independent, symmetrically and identically distributed functions such that $g_1\in L_{2,q}\setminus L_2$, then
$$
\frac{1}{\sqrt{n}}\left\|\sum_{k=1}^n g_k\right \|_{L_{2,q}}\to\infty\;\;\mbox{as}\;\;n\to\infty.
$$
Since $\{g_k\}$ is a symmetric basic sequence, as in the preceding case,
by \cite[Corollary~3.2]{Arazy81}, we can find a sequence $\{x_k\}\subset \ell_2\oplus \ell_{2,q}$ such that for all $a_k\in\mathbb{R}$
\begin{equation*}
\label{eq3}
\left\|\sum_{k=1}^\infty a_kx_k\right\|_{\ell_2\oplus \ell_{2,q}}\asymp \left\|\sum_{k=1}^\infty a_kg_k\right\|_{L_{2,q}}.
\end{equation*}
Then, clearly,
\begin{equation}
\label{eq3}
\frac{1}{\sqrt{n}}\left\|\sum_{k=1}^n x_k\right\|_{\ell_2\oplus \ell_{2,q}}\to\infty\;\;\mbox{as}\;\;n\to\infty.
\end{equation}

Since $g_k$, $k=1,2,\dots$, are independent, symmetrically and identically distributed, then the sequence $\{g_k\}$ in $L_{2,q}$ is equivalent to the sequence $\{g_k(s)r_k(t)\}$ in $L_{2,q}([0,1]\times [0,1])$ 
(as above, $r_k$ are the Rademacher functions). One can readily check that the latter sequence is weakly null in $L_{2,q}([0,1]\times [0,1])$;
 so is $\{g_k\}$ in $L_{2,q}$. Thus, $\{x_k\}$ is weakly null in $\ell_2\oplus \ell_{2,q}$, and, as above, we may assume that $x_k$, $k=1,2,\dots,$ are pairwise disjoint.

If $x_k=y_k+z_k$, where $y_k\in \ell_2$ and $z_k\in \ell_{2,q}$, $k=1,2,\dots$, then we have
\begin{equation*}
\label{eq4}
\left\|\sum_{k=1}^\infty a_kx_k\right\|_{\ell_2\oplus \ell_{2,q}}\asymp \left\|\sum_{k=1}^\infty a_ky_k\right\|_{\ell_2}+\left\|\sum_{k=1}^\infty a_kz_k\right\|_{\ell_{2,q}}.
\end{equation*}
Thus, since the sequences $\{y_k\}$ and $\{z_k\}$ consist of pairwise disjoint elements and the space $\ell_{2,q}$, $q>2$, admits an upper $2$-estimate (see e.g. \cite[Theorem~3]{Dilworth}), this inequality yields
$$
\left\|\sum_{k=1}^\infty a_kx_k\right\|_{\ell_2\oplus \ell_{2,q}}\le C'\Big\{\Big(\sum_{k=1}^\infty |a_k|^2\|y_k\|_{\ell_2}^2\Big)^{1/2}+\Big(\sum_{k=1}^\infty |a_k|^2\|z_k\|_{\ell_{2,q}}^2\Big)^{1/2}\Big\}\le C\left\| (a_k)\right \|_{\ell_2},$$
which contradicts \eqref{eq3}.
\end{proof}

Summing up Theorems \ref{Lor}, \ref{theor:easy} and Proposition \ref{prop easy}, we get the following result.

\begin{theorem}
If a couple $(p,q)$ of positive numbers satisfies one of the following conditions:
 \begin{enumerate}
   \item $1< p < 2$  and  $1\le q<\infty $;
   \item $p=2$ and $q\in [1,2)\cup (2,\infty)$;
   \item $2<q\le p<\infty$,
 \end{enumerate}
 then we have
$$L_{p,q}(0,1) \not\hookrightarrow C_{p,q}.  $$
\end{theorem}

\begin{quest}
Does $L_{p,q}(0,1)$ isomorphically embed into $C_{p,q}$ when
 $2<p<\infty$ and  $q\in [1,2]\cup (p,\infty)$?
\end{quest}


%


%

\end{document}